\documentclass[reqno,11pt,twoside]{amsart}
\usepackage{amssymb,graphicx}
\usepackage{fullpage}
\numberwithin{equation}{section}

\usepackage{xypic} 
\usepackage{mathrsfs}
\usepackage{hyperref}
\hypersetup{ 
colorlinks,
citecolor=black,
filecolor=blue,
linkcolor=black,
urlcolor=blue 
} 
\usepackage{comment}
\usepackage{color}
\xyoption{curve}

\numberwithin{equation}{subsection}
\newtheorem{lemma}[equation]{Lemma}
\newtheorem{lem/def}[equation]{Lemma/Definition}
\newtheorem{proposition}[equation]{Proposition}
\newtheorem{theorem}[equation]{Theorem}
\newtheorem{corollary}[equation]{Corollary}

\theoremstyle{definition}
\newtheorem{example}[equation]{Example}
\newtheorem{definition}[equation]{Definition}

\newtheorem{hypotheses}[equation]{Hypotheses}
\newtheorem*{construction_pg}{Construction of $p_g$}
\newtheorem{observation}[equation]{Observation}

\newcommand{\ba}{\begin{array}}
\newcommand{\bu}
{\bullet}

\newcommand{\ea}{\end{array}}
\newcommand{\Hom}{\mathrm{Hom}}
\newcommand{\K}{\widetilde{K}}

\newcommand{\Spec}{\mathrm{Spec}}
\newcommand{\ot}{\otimes}
\newcommand{\ox}{\otimes}

\newcommand{\xym}{\xymatrix}
\newcommand{\codim}{\mathrm{codim}}
\newcommand{\unl}{\underline}

\newcommand{\Wedge}{\textstyle\bigwedge} 

\DeclareMathOperator{\coh}{H}
\DeclareMathOperator{\HH}{HH}
\DeclareMathOperator{\id}{id}
\DeclareMathOperator{\im}{Im}
\DeclareMathOperator{\sgn}{sgn}

\title[Lie bracket on Hochschild cohomology]{The Gerstenhaber bracket as a Schouten bracket for polynomial rings extended by finite groups}
\date{June 1, 2017}

\author{Cris Negron}
\address{Department of Mathematics, Massachusetts Institute of Technology\\ Cambridge, MA 02139, USA}
\email{negronc@mit.edu}

\author{Sarah Witherspoon}
\address{Department of Mathematics\\Texas A\&M University\\College Station, TX 77843,
USA}
\email{sjw@math.tamu.edu}
\thanks{The first author was supported by NSF Postdoctoral fellowship DMS-1503147.
The second author was supported by NSF grant DMS-1401016.}

\begin{document}

\maketitle
\begin{abstract}
We apply new techniques to compute Gerstenhaber brackets on the Hochschild cohomology of a skew group algebra formed from a polynomial ring and a finite group (in characteristic 0). We  show that the Gerstenhaber brackets can always be expressed in terms of Schouten brackets on polyvector fields.  We obtain as consequences some conditions under which brackets are always 0, and show that the Hochschild cohomology is a graded Gerstenhaber algebra under the codimension grading, strengthening known results. 
\end{abstract}

\section{Introduction}

We compute  brackets on the Hochschild cohomology of
a skew group algebra formed from a symmetric algebra (i.e.\ polynomial ring)
and a finite group in characteristic~0.
Our results strengthen those given by Halbout and Tang~\cite{Halbout-Tang10} and by 
Shepler and the second author~\cite{SW2}: 
The Hochschild cohomology decomposes as a direct sum 
indexed by conjugacy classes of the group.
In~\cite{Halbout-Tang10,SW2} the authors give  formulas for Gerstenhaber brackets in terms of this
decomposition, compute examples, and present vanishing results. 
Here we go further and show  that  brackets are
always  sums of projections of Schouten brackets onto these group components.
As just one consequence, a bracket of two nonzero 
cohomology classes supported {\em off} the kernel of the
group action is always 0 when their homological degrees are smallest possible.
Our results complete the picture begun in \cite{Halbout-Tang10,SW2}, facilitated here
by new techniques from~\cite{Negron-Witherspoon}.

From a theoretical perspective, the category of modules over the skew group algebra under consideration here is equivalent to the category of equivariant quasi-coherent sheaves on the corresponding affine space.  So the Hochschild cohomology of the skew group algebra, along with the Gerstenhaber bracket, is reflective of the deformation theory of this category.  This relationship can be realized explicitly through the work of Lowen and Van den Bergh \cite{lowenvandenbergh05}.  The Hochschild cohomology of the skew group algebra is also strongly related to Chen and Ruan's orbifold cohomology (see e.g. \cite{dolgushevetingof05,pflaumetal11}). 

We briefly summarize our main results.
If $V$ is a finite dimensional vector space with an action of a finite group
$G$, there is an induced action of $G$ by automorphisms on the symmetric
algebra $S(V)$, and one may form the skew group algebra
(also known as a smash product or semidirect product)  $S(V)\# G$.  Its Hochschild cohomology $H:=\HH^\bu(S(V)\# G)$ is isomorphic to 
the $G$-invariant subspace of a direct sum $\oplus_{g\in G} H_g$, and the
$G$-action permutes the components via the conjugation action of $G$ on itself.
See for example~\cite{farinati,ginzburgkaledin04,NPPT}; 
we give some details as needed in Section~\ref{cohomcomp}. 
\par

Each space $H_g$ may be viewed in a canonical way as a subspace of 
$S(V)\ot \Wedge^\bu V^*$, and we construct canonical projections $p_g:S(V)\ot \Wedge^\bu V^\ast\to H_g$.  Since the space $S(V)\ot \Wedge^\bu V^*$ can be identified with the algebra of polyvector fields on affine space, it admits a canonical graded Lie bracket, namely the Schouten bracket
(also known as the Schouten-Nijenhuis bracket), which we denote here by $\{\ ,\ \}$.  By way of the inclusions $H_g\subset S(V)\ot \Wedge^\bu V^\ast$ we may apply the Schouten bracket to elements of $H_g$, and hence to elements in the Hochschild cohomology $\HH^\bu(S(V)\# G)=(\oplus_{g\in G} H_g)^G$.  Our main theorem is:

\medskip

\noindent
{\bf Theorem \ref{thm:gh}}. {\em Let $X =\sum_{g\in G}X_g$ and $Y=\sum_{h\in G}Y_h$
be classes in $\HH^{\bu}(S(V)\# G)$ where $X_g\in H_g$, $Y_h\in H_h$. Their Gerstenhaber bracket is }
\[
    [X,Y] = \sum_{g,h\in G}  p_{gh} \{ X_g,Y_h\}.
\]

\smallskip

This result was obtained by Halbout and Tang~\cite[Theorem~4.4, Corollary~4.11]{Halbout-Tang10} in
some special cases. 

In the body of the text, we assign to each summand $H_g$ its own copy of $S(V)\ot \Wedge^\bu V^\ast$ and label this copy with a $g$.  So we will write instead $S(V)\ot \Wedge^\bu V^\ast g$, and write elements in $S(V)\ot \Wedge^\bu V^\ast g$ as $X_gg$ instead of just $X_g$.

We note that all of the ingredients in the expressions $\sum p_{gh} \{ X_g,Y_h\}$ are canonically defined.  Thus we have closed-form expressions for Gerstenhaber brackets on classes in arbitrary degree.  There are very few  algebras for which we have such an understanding of the graded Lie structure on Hochschild cohomology, aside from smooth commutative algebras (over a given base field).  For a smooth commutative algebra $R$ there is the well known HKR isomorphism~\cite{HKR} between the Hochschild cohomology of $R$ and polyvector fields on $\mathrm{Spec}(R)$, along with the Schouten bracket.  The particular form of Theorem \ref{thm:gh} is referential to this classic result.  Having such a complete understanding of the Lie structure is useful, for example, in the production of $L_\infty$-morphisms and formality results (see \cite{ginothalbout03,MK}).

The proof of Theorem~\ref{thm:gh} uses the  approach to 
Gerstenhaber brackets given in~\cite{Negron-Witherspoon}, in 
which we introduced new techniques that 
are particularly well-suited to  computations.  
We summarize the necessary material from~\cite{Negron-Witherspoon} in Section~\ref{sec:prelim}, explaining how to apply it to skew group algebras. 
In Sections~\ref{polyrings} and \ref{sec:circ-proj}, we develop further 
the theory needed to apply the techniques to the skew group algebra $S(V)\# G$
in particular. 

We obtain a number of consequences of Theorem~\ref{thm:gh} in 
Sections~\ref{sec:brackets} and \ref{sec:nonvanishing}. 
In Corollary~\ref{cor:brackform0} we recover \cite[Corollary 7.4]{SW2},
stating that in case $X,Y$ are supported entirely on group elements
acting trivially on $V$, the Gerstenhaber bracket is simply the
sum of the componentwise Schouten brackets. 
In Corollary~\ref{cor:vanish1} we recover \cite[Proposition 8.4]{SW2}, giving
some conditions on invariant subspaces under which the bracket $[X,Y]$
is known to vanish.
Corollary~\ref{thm:supp-off-id} is another vanishing result that generalizes
\cite[Theorem 9.2]{SW2} from degree 2 to arbitrary degree, stating that
in case $X,Y$ are supported entirely off the kernel of the group action
and their homological degrees are smallest possible, their 
Gerstenhaber bracket is 0.  In Corollary~\ref{cor:codimgrading} we also show that the Hochschild cohomology $\HH^\bu(S(V)\# G)$ is a graded Gerstenhaber algebra with respect to a certain natural grading coming from the geometry of the fixed spaces $V^g$, which we refer to as the codimension grading.
Section~\ref{sec:nonvanishing} consists of examples and a general
explanation of (non)vanishing of the Gerstenhaber bracket for $S(V)\#G$,
rephrasing some of the results of \cite{SW2}.

Let $k$ be a field and $\ot = \ot_k$.
For our main results, we assume the characteristic of $k$ is 0, 
but this is not needed for the general techniques presented in Section~\ref{sec:prelim}.  We adopt the convention that a group $G$ will act on the left of an algebra, and on the right of functions from that algebra.  Similarly, we will let $G$ act on the left of the finite dimensional vector space 
$V\subset S(V)$ and on the right of its dual space $V^\ast$.

\section{An alternate approach to the Lie bracket}\label{sec:prelim} 

Let $A$ be an algebra over the field $k$. 
Let $B\rightarrow A$ denote the {\em bar resolution} of $A$ as an $A$-bimodule:
\[
    \cdots \stackrel{\delta_3}{\longrightarrow} A^{\ot 4}
\stackrel{\delta_2}{\longrightarrow} A^{\ot 3}
\stackrel{\delta_1}{\longrightarrow} A\ot A \stackrel{\mu}{\longrightarrow}
  A\rightarrow 0 ,
\]
where $\mu$ denotes multiplication and 
\begin{equation}\label{eqn:bar-diff}
\delta_n(a_0\ot \cdots\ot a_{n+1})
= \sum_{i=0}^{n} (-1)^i a_0\ot\cdots\ot a_i a_{i+1}\ot \cdots\ot a_{n+1}
\end{equation}
for all $a_0,\ldots,a_{n+1}\in A$. 
The {\em Gerstenhaber bracket} of homogeneous functions $f,g\in \Hom_{A^e}(B,A)$
is defined to be
\[
    [f,g] = f\circ g - (-1)^{(|f|-1)(|g|-1)} g\circ f,
\]
where the circle product $f\circ g$ is given by 
\[
\begin{aligned}
&   (f\circ g) (a_1\ot\cdots\ot a_{|f|+|g|-1}) \\
 & \hspace{.2cm} = \sum_{j=1}^{|f|} (-1)^{(|g|-1)(j-1)} 
  f(a_1\ot \cdots \ot a_{j-1}\ot g(a_j\ot\cdots\ot a_{j+|g|-1})\ot 
  \cdots\ot a_{|f|+|g|-1}), 
\end{aligned}
\]
for all $a_1,\ldots, a_{|f|+|g|-1}\in A$, 
and similarly $g\circ f$.
This induces the bracket on Hochschild cohomology. 
(Here we have identified $\Hom_{A^e}(B,A)$ with $\Hom_k (A^{\ot \bu}, A)$.)

Typically when one computes brackets, one uses explicit chain maps 
between this bar resolution $B$ and a more convenient one for computational
purposes, navigating back and forth.
Such chain maps are usually awkward, and this way can be inefficient
and technically difficult. 
In this section, we first recall from \cite{Negron-Witherspoon} an alternate
approach, for some types of algebras, introduced to avoid this trouble. 
Then we explain
how to apply it  to skew group algebras in particular.

\subsection{A collection of brackets} 
Given a bimodule resolution $K\to A$ satisfying some conditions as detailed below, one can produce a number of coarse brackets $[\ ,\ ]_\phi$ on the complex $\Hom_{A^e}(K,A)$, each depending on a map $\phi$.  These brackets are coarse in the sense that they will not, in general, produce dg Lie algebra structures on the complex $\Hom_{A^e}(K,A)$.  They will, however, be good enough to compute the Gerstenhaber bracket on the cohomology $\coh^{\bu}(\Hom_{A^e}(K,A))=\HH^{\bu}(A)$.  We have precisely: 

\begin{theorem}[{\cite[Theorem 3.2.5]{Negron-Witherspoon}}]\label{thm:NW} 
Suppose $\mu:K\to A$ is a projective $A$-bimodule resolution of $A$ satisfying the Hypotheses~\ref{hypotheses}(a)--(c) below.  Let $F_K:K\ox_AK\to K$ be the chain map $F_K=\mu\ox \id_K-\id_K\ox \mu$.  Then for any degree $-1$ bimodule map $\phi:K\ox_AK\to K$ satisfying $d_K\phi+\phi d_{K\ox_A K} = F_K$, there is a bilinear operation $[\ ,\ ]_\phi$ on $\Hom_{A^e}(K,A)$.  Each operation $[\ ,\ ]_\phi$ satisfies the following properties. 
\begin{enumerate}
\item $[f,g]_\phi$ is a cocycle whenever $f$ and $g$ are cocycles in $\Hom_{A^e}(K,A)$. 
\item $[f,g]_\phi$ is a coboundary whenever, additionally, $f$ or $g$ is a coboundary. 
\item The induced operation on cohomology $\coh^{\bu}(\Hom_{A^e}(K,A))=\HH^{\bu}(A)$ is precisely the Gerstenhaber bracket:
$ \ 
[f,g]_\phi=[f,g]
$ on cohomology. 
\end{enumerate}
\end{theorem}

The bilinear map $[ \ , \ ]_{\phi}$ is defined by equations 
(\ref{eqn:phi-circle-def}) and (\ref{eqn:phi-bracket-def}) below. 
By \cite[Lemma 3.2.1]{Negron-Witherspoon}, such maps $\phi$ satisfying the conditions in the theorem  always exist.  
We call such a map $\phi$ a {\it contracting homotopy} for $F_K$.  

There is a {\em diagonal map} 
$\Delta_B: B\rightarrow B\ot _A B$ given by 
$$
\Delta_B(a_0\ox\dots\ox a_{n+1})=\sum^n_{i=0}(a_0\ox\dots\ox a_i\ox 1)\ox_A(1\ox a_{i+1}\ox\dots\ox a_{n+1})
$$
for all $a_0,\ldots,a_{n+1}\in A$. 
The  hypotheses on the projective $A$-bimodule resolution $K\rightarrow A$,
to which the theorem above  refers, are as follows.

\begin{hypotheses}\label{hypotheses}
\begin{enumerate}
\item[(a)] $K$ admits a chain embedding $\iota:K\to B$  
that fits into a commuting diagram
$$
\xymatrixrowsep{4mm}
\xym{
K\ar[rr]^\iota\ar[dr] & & B\ar[dl]\\
 & A &
}
$$
\item[(b)] The embedding $\iota$ admits a retract $\pi$.  That is, 
there is an chain map $\pi : B\rightarrow K$ such that $\pi \iota=\id_K$.
\item[(c)] The diagonal map $\Delta_B:B\to B\ox_A B$ preserves $K$, and hence restricts to a diagonal map  $\Delta_K:K\to K\ox_A K$.  Equivalently, $\Delta_B \iota=(\iota\ox_A \iota)\Delta_K$.
\end{enumerate}
\end{hypotheses}

Conditions (a) and (c) together can alternatively be stated as the condition that $K$ is a dg coalgebra in the monoidal category $A$-bimod that admits an embedding into the bar resolution.  We denote the coproducts of  elements in $K$ using Sweedler's notation, e.g.\ 
$$
\Delta(w)=w_1\ox w_2,\ \ \ (\Delta\ox_A \id_K)\Delta(w)=w_1\ox w_2\ox w_3,
$$
for $w\in K$,
with the implicit sum $\Delta(w)=\sum_{i_1, i_2} w_{i_1}\ox w_{i_2}$ suppressed.  Koszul resolutions of Koszul algebras, as well as related algebras such as universal enveloping algebras, Weyl algebras, and Clifford algebras, will fit into this framework \cite{BS,N}.
\par

Given a  resolution $K\to A$ satisfying  Hypotheses~\ref{hypotheses}(a)--(c)
and a contracting homotopy $\phi$ for $F_K$, we construct the $\phi$-bracket as follows: We first define the {\em $\phi$-circle product} for functions $f,g$ in
$\Hom_{A^e}(K,A)$ by
\begin{equation}\label{eqn:phi-circle-def}
(f\circ_\phi g)  (w)=(-1)^{|w_1||g|}f\big(\phi(w_1\ox g(w_2)\ox w_3)\big)
\end{equation}
for homogeneous $w$ in $K$, 
and define the {\em $\phi$-bracket} as the graded commutator
\begin{equation}\label{eqn:phi-bracket-def}
[f,g]_\phi=f\circ_\phi g-(-1)^{(|f|-1)(|g|-1)}g\circ_\phi f.
\end{equation} 
\par

We will be interested in producing such brackets particularly for a 
skew group algebra formed from a symmetric algebra (i.e.\ polynomial ring) 
under a finite group action.
We will start with the Koszul resolution $K$ 
of the symmetric algebra itself and then 
construct a natural extension $\widetilde{K}$ to resolve the skew group algebra.  
We will define a contracting homotopy $\phi$ for $F_K$ that extends  to $\widetilde{K}$, 
and use it to compute Gerstenhaber brackets via 
Theorem~\ref{thm:NW}. 
We describe these constructions next in the context of more
general skew group algebras. 


\subsection{Skew group algebras}
Let $G$ be a finite group whose order is not divisible by the 
characteristic of the field $k$. Assume that  $G$ acts by automorphisms on the algebra $A$.
Let $B=B(A)$ be the bar resolution of $A$, and 
let $K=K(A)$ be a projective $A$-bimodule resolution of $A$ satisfying Hypotheses~\ref{hypotheses}(a)--(c).
Assume that $G$ acts on $K$ and on $B$,
and this action commutes with the differentials and with
the maps $\iota, \pi, \Delta_K, \Delta_B$. 
These assumptions all hold in the case that $A$ is a Koszul algebra 
on which $G$ acts by graded automorphisms 
and $K$ is a Koszul resolution; in particular, $\pi$ may be replaced by 
$\frac{1}{|G|} \sum_{g\in G} g \pi g^{-1}$ if it    is not a priori $G$-linear. 

Let $A\# G$ denote the {\em skew group algebra},
that is $A\ot kG$ as a vector space, with multiplication  defined by
$(a\ot g)(b\ot h) = a ( {}^gb)\ot gh$ for all $a,b\in A$ and $g,h\in G$,
where left superscript denotes the $G$-action. 
We will sometimes write $a\# g$ or simply $ag$ in place of 
the element $a\ot g$ of $A\# G$ when
there can be no confusion. 

Let $B(A\# G)$ denote the bar resolution of $A\# G$ as a $k$-algebra, and let 
$\widetilde{B}(A\# G)$ denote its bar resolution {\em over} $kG$: 
$$
   \cdots \stackrel{\delta_3}{\longrightarrow}
  (A\# G)^{\ot_{kG} 4} \stackrel{\delta_2}{\longrightarrow}
  (A\# G)^{\ot_{kG} 3} \stackrel{\delta_1}{\longrightarrow}
  (A\# G)\ot_{kG} (A\# G) \stackrel{\mu}{\longrightarrow}
  A\# G \rightarrow 0 ,
$$
with differentials defined as in (\ref{eqn:bar-diff}).
Since $kG$ is semisimple, this is a projective resolution of $A\# G$
as an $(A\# G)^e$-module.
There is a vector space isomorphism 
$$
   (A \# G)^{\ot_{kG} i}\cong A^{\ot i} \ot kG
$$
for each $i$, and from now on we will identify $\widetilde{B}_j(A\# G)$ with 
$A^{\ot (j+2)} \ot kG$ for each $j$, and the differentials of $\widetilde{B}$ 
with those of $B$ tensored with the identity map on $kG$.

Similarly we wish to extend $K$ to a projective resolution of $A\# G$ as an $(A\# G)^e$-module.
Let $\widetilde{K}= \widetilde{K}(A\# G)$ denote the following complex: 
$$  \cdots \stackrel{d_3\ot\id_{kG}}{\relbar\joinrel\relbar\joinrel\relbar\joinrel\relbar\joinrel\longrightarrow}
  K_2\ot kG \stackrel{d_2\ot\id_{kG}}{\relbar\joinrel\relbar\joinrel\relbar\joinrel\relbar\joinrel\longrightarrow}
  K_1\ot kG \stackrel{d_1\ot\id_{kG}}{\relbar\joinrel\relbar\joinrel\relbar\joinrel\relbar\joinrel\longrightarrow}
  K_0\ot kG \stackrel{\mu\ot\id_{kG}}{\relbar\joinrel\relbar\joinrel\relbar\joinrel\longrightarrow}
  A\# G \rightarrow 0 .
$$
We give the terms of this complex the structure of $A\# G$-bimodules as follows:
\begin{eqnarray*}
   (a\# g) (x\ot h) & = &  a( {}^gx)\ot gh , \\
   (x\ot h) (a\# g) &=& x ({}^ha)\ot hg , 
\end{eqnarray*}
for all $a\in A$, $g,h\in G$, and $x\in K_i$. 
Then $\widetilde{K}$ is a projective resolution of $A\# G$ by 
$(A\# G)^e$-modules. 

Next we will show that $\widetilde{K}\rightarrow A\# G$ satisfies 
Hypotheses~\ref{hypotheses}(a)--(c) for the algebra $A\# G$.
Let $\tilde{\iota} : \widetilde{K} \rightarrow B(A\# G)$ be the composition
$$
   \widetilde{K} \stackrel{ \iota\ot \id_{kG}}
  {\relbar\joinrel\relbar\joinrel\relbar\joinrel\longrightarrow}
  \widetilde{B}(A\# G) \stackrel{i}{\longrightarrow} B(A\# G)
$$
where $i(a_0\ot \cdots\ot a_{j+1}\ot g) = (a_0\# 1)\ot \cdots\ot (a_j\# 1)\ot (a_{j+1}\# g)$
for all $a_0,\ldots, a_{j+1} \in A$ and $g\in G$. 
Let $\tilde{\pi} : B(A\# G)\rightarrow \widetilde{K}$ be the composition
$$
  B(A\# G) \stackrel{p}{\longrightarrow} \widetilde{B}(A\# G) \stackrel
   {\pi\ot \id_{kG} }{\relbar\joinrel\relbar\joinrel\relbar\joinrel\longrightarrow}
   \widetilde{K} 
$$
where 
\begin{equation}\label{eqn:move-g-right}
  p((a_0\# g_0)\ot (a_1\# g_1)\ot (a_2\# g_2) \ot\cdots\ot (a_j\# g_j))
   \hspace{3cm}
\end{equation}
\[
 \hspace{1cm} = a_0\ot ( {}^{g_0}a_1) \ot ( {}^{g_0g_1}a_2) \ot\cdots\ot ({}^{g_0g_1\cdots g_{j-1}} a_j)
\ot g_0g_1\cdots g_j
\]
for all $a_0,\ldots,a_j\in A$ and $g_0,\ldots, g_j\in G$. 
One can check that $i$ and $p$ are indeed chain maps.
Then
$$
  \tilde{\pi}\tilde{\iota} = (\pi\ot \id_{kG}) p  i (\iota\ot \id_{kG}) =
    \id _{\widetilde{K}}
$$
since $pi = \id_{\widetilde{B}(A\# G)}$ by the definitions of these maps, and 
$\pi\iota = \id_{K}$ by Hypothesis~\ref{hypotheses}(b)  applied to $K$. 
Therefore Hypotheses~\ref{hypotheses}(a) and (b) hold for
$\widetilde{K}$.

Now let $\Delta_{\widetilde{K}} : \widetilde{K} \rightarrow
\widetilde{K}\ot _{A\# G} \widetilde{K}$ be defined by $\Delta_{\widetilde{K}}=
\Delta_{K}\ot \id_{kG}$, after identifying
 $
   \widetilde{K}_i\ot_{A\# G} \widetilde{K}_j =
  (K_i\ot kG) \ot_{A\# G} (K_j\ot kG) $ 
with
$ (K_i\ot _AK_j)\ot kG$. 
One may check that $\Delta_{\widetilde{K}}$ satisfies Hypothesis~\ref{hypotheses}(c). 

Let $\phi_K: K\ot_A K\rightarrow K$ be a map satisfying $d_K \phi_K + \phi_K d_{K\ot_AK} = F_K$ as in  Theorem~\ref{thm:NW}.
Assume that $\phi_K$ is  $G$-linear.
Let 
\begin{equation}\label{eqn:phiKtilde}
   \phi_{\K}= \phi_{K}\ot\id_{kG},
\end{equation} 
 a map from $\widetilde{K}
\ot _{A\# G}\K$ to $\K$,  under the identification
$\K  \ot_{A\# G} \K \cong (K\ot _A K)\ot kG$. 
Since $\phi_K$ is $G$-linear, this map $\phi_{\widetilde{K}}$ is an $A\# G$-bimodule map.
Further,  
$$
   d_{\widetilde{K}}\phi_{\K} + \phi_{\K}d_{\K\ot_A \K} 
= (d_K\phi_{K} + \phi_K d_{K\ot_A K} )\ot \id_{kG} = F_{K} \ot \id_{kG} = F_{\K}.
$$
As a consequence, 
by Theorem~\ref{thm:NW},  $\phi_{\widetilde{K}}$  may be used to define the
Gerstenhaber bracket on the Hochschild cohomology of $A\# G$ via (\ref{eqn:phi-circle-def}) and (\ref{eqn:phi-bracket-def}).

\section{Symmetric algebras}
\label{polyrings}

Let $V$ be a finite dimensional vector space over the field $k$
of characteristic 0, 
and let $A=S(V)$, the symmetric algebra on $V$. 
In this section, we construct a map $\phi$ that will allow us to 
compute Gerstenhaber brackets on the Hochschild cohomology of  
the skew group algebra $A\# G$ 
arising from a representation of the finite group $G$ on $V$,
via Theorem~\ref{thm:NW} and equation~(\ref{eqn:phiKtilde}).  

\subsection{The Koszul resolution}
We will use a standard
description of the Koszul resolution $K$ of $A=S(V)$ as an $A^e$-bimodule, given as a
subcomplex of the bar resolution $B=B(A)$: For all $v_1,\ldots,v_i\in V$, let 
$$
   o(v_1,\ldots, v_i)=\sum_{\sigma\in S_i} \sgn (\sigma) v_{\sigma(1)}\ot\cdots
    \ot v_{\sigma(i)} 
$$
in $V^{\ot i} \subset A^{\ot i}$. 
Sometimes we write instead
\[
   o(v_I) = o(v_1,\ldots, v_i),
\]
where $I=\{1,\ldots,i\}$. 
Take $o(\varnothing) =1$.
Let $K_i=K_i(A)$ be the free $A^e$-submodule of $A^{\ot (i+2)}$
with free basis all $1\ot o(v_1,\ldots,v_i)\ot 1$ in $A^{\ot (i+2)}$. 
We may in this way identify $K_i$ with $A\ot\Wedge^iV\ot A$ and 
$K$ with $A\ot \Wedge^{\bu} V\ot A$.
The differentials on the bar resolution, restricted to $K$, may be rewritten
in terms of the chosen free basis of $K$ as
\[
\begin{aligned}
&  d(1\ot o(v_1,\ldots,v_i)\ot 1) \\ &\hspace{.5cm} =
    \sum_{j=1}^i (-1)^{j-1}(v_j\ot o(v_1,\ldots, \hat{v}_j,\ldots, v_i)\ot 1
   - 1\ot o(v_1,\ldots, \hat{v}_j, \ldots, v_i) \ot v_j) ,
\end{aligned}
\]
where $\hat{v}_j$ indicates that $v_j$ has been deleted from the list of vectors.
Note that the action of $G$ preserves the vector subspace of $K_i$ spanned
by all $1\ot o(v_1,\ldots,v_i)\ot 1$. 
There is a retract $\pi: B\rightarrow K$ that can be chosen to be $G$-linear.
We will not need an explicit formula for $\pi$ here.

The diagonal map $\Delta_{K} : K\to K\ox_A K$ is given by 
\begin{equation}
\Delta_K (1\ox o(v_I)\ox 1)=\sum_{I_1,I_2} \pm (1\ox o(v_{I_1})\ox 1)\ox(1\ox o(v_{I_2})\ox 1)
\label{cmltp}
\end{equation}
where the sum is indexed by all ordered disjoint subsets $I_1,I_2\subset I$
with $I_1\cup I_2=I$, and $\pm$ is the sign of $\sigma$, 
the unique permutation for which $\{i_{\sigma(1)},\dots
i_{\sigma(|I|)}\}=I_1\cup I_2$ as an ordered set.
Note that $\Delta_K$ is $G$-linear. 
Letting $\iota: K\rightarrow B$ be the embedding of $K$ as a subcomplex of $B$,
Hypotheses~\ref{hypotheses}(a)--(c) now hold.
We may thus use Theorem~\ref{thm:NW} and (\ref{eqn:phiKtilde}) 
to compute brackets on $\HH^{\bu}(S(V)\# G)$ 
once we have a suitable map $\phi$. 

\subsection{An invariant map $\phi$}
We will now define an $A$-bilinear map $\phi: K\ox_A K \to K$ that will be 
$G$-linear and independent of choice of ordered basis of $V$. Compare with 
\cite[Definition~4.1.3]{Negron-Witherspoon}, where a simpler map $\phi$ was defined
which however depends on such a choice. 
At first glance, the map $\phi$ below looks rather complicated, 
but in practice we find it easier to use when extended to a skew group algebra
than  explicit chain maps between bar and Koszul resolutions.

\begin{definition}\label{def:phi}
{\sl Assume the characteristic of $k$ is 0. Let
$\phi:K\ox_A K\to K$ be the $A$-bilinear map given on ordered
monomials as follows: 
\\
\[
   \phi_0(1\ot v_1\cdots v_t\ot 1) =
   \frac{1}{t!} \sum_{\substack{1\leq r\leq t \\ \sigma\in S_t}}
    v_{\sigma(1)}\cdots v_{\sigma(r-1)} \ot o(v_{\sigma(r)})\ot
   v_{\sigma(r+1)}\cdots v_{\sigma(t)} 
\]
for all $v_1,\ldots, v_t\in V$. 
On $K_0\ot_A K_z$ and on $K_s\ot _A K_0$ ($s,z >0$), let 
\[ 
\begin{aligned}
&\phi_z(1\ot v_1\cdots v_t\ot o(w_1,\ldots, w_z)\ot 1) \\  &=
   \frac{(-1)^{z}}{(t+z)!} \sum_{\substack{1\leq r\leq t\\ \sigma\in S_t}}
    \big(\prod_{i=0}^{z-1} (r+i)\big)
    v_{\sigma(1)}\cdots v_{\sigma(r-1)}
    \ot o(w_1,\ldots, w_z,v_{\sigma(r)}) \ot
    v_{\sigma(r+1)}\cdots v_{\sigma(t)} ,\\
&\phi_s(1\ot o(u_1,\ldots, u_s) \ot v_1\cdots v_t \ot 1 \\ &=
   \frac{1}{(t+s)!} \sum_{\substack{1\leq r\leq t\\ \sigma\in S_t}}
    \big(\prod_{i=1}^s (t-r+i)\big) 
   v_{\sigma(1)}\cdots v_{\sigma(r-2)}
   \ot o(v_{\sigma(r)}, u_1, \ldots, u_s)\ot v_{\sigma(r+1)}\cdots
    v_{\sigma(t)} , 
\end{aligned} 
\] 
for all $u_1,\ldots, u_s,v_1,\ldots, v_t, w_1,\ldots, w_z\in V$.
When $s>0$ and $z>0$, let
\[
\begin{aligned} 
&\phi_{s+z}(1\ot o(u_1,\ldots, u_s)\ot v_1\cdots v_t \ot o(w_1,\ldots, w_z)\ot 1) \\
&= 
\sum_{\substack{1\leq r\leq t\\ \sigma\in S_t}}
 c^{s,t,z}_r v_{\sigma(1)}\cdots v_{\sigma(r-1)}\ot o(w_1,\ldots,w_z,v_{\sigma(r)},
   u_1,\ldots, u_s)\ot v_{\sigma(r+1)}\cdots v_{\sigma(t)},
\end{aligned}
\]
where $ \ \displaystyle{ c^{s,t,z}_r = \frac{(-1)^{sz+z}}{(s+t+z)!} 
       \big(\prod_{i=0}^{z-1} (r+i)\big) \big(\prod_{j=1}^s (t-r+j)\big)}$. 
}\end{definition}

To see that $\phi$ is well-defined, one can first construct the corresponding map from the tensor powers $(V^{\ot s})\ot (V^{\ot t})\ot (V^{\ot z})$ (using the universal property of the tensor product, for example) then note that the given map is $S_s\times S_t\times S_z$-invariant and hence induces a well-defined map $\phi$ on the coinvariants $(\Wedge^s V)\ot S^t(V)\ot (\Wedge^z V)$.  One can also see directly that $\phi$ is $G$-invariant.  In fact, it is invariant under the action of the entire group $\mathrm{GL}(V)$.  We next state that $\phi$ is a contracting homotopy for the map $F_K$ defined in the statement of Theorem~\ref{thm:NW}.

\begin{lemma}\label{lemma:phimap}
Let $\phi: K\ot_A K\rightarrow K$ be the $A$-bilinear map of 
Definition~\ref{def:phi}. Then 
$ \ d_K\phi + \phi d_{K\ot_AK} = F_K$.
\end{lemma}

\begin{proof}
In degree 0, we check:
\[
\begin{aligned}
  & (d\phi+\phi d) (1\ot v_1\cdots v_t \ot 1) \\
  &= d\left(\frac{1}{t!} \sum_{\substack{1\leq r\leq t\\ \sigma\in S_t}}
    v_{\sigma(1)}\cdots v_{\sigma(r-1)} \ot o(v_{\sigma(r)})\ot v_{\sigma(r+1)}
   \cdots v_{\sigma(t)}\right)\\
  &= \frac{1}{t!} \sum_{\substack{1\leq r\leq t\\ \sigma\in S_t}}
  (v_{\sigma(1)}\cdots v_{\sigma(r)} \ot v_{\sigma(r+1)}
   \cdots v_{\sigma(t)}  -  v_{\sigma(1)}\cdots v_{\sigma(r-1)} \ot v_{\sigma(r)}
   \cdots v_{\sigma(t)}) \\
   & =\frac{1}{t!}\sum_{\sigma\in S_t} ( v_{\sigma(1)}\cdots v_{\sigma(t)}\ot 1
   - 1\ot v_{\sigma(1)}\cdots v_{\sigma(t)} )\\
  & = v_1\cdots v_t \ot 1 - 1\ot v_1\cdots v_t \ \ = \ \ F_K(1\ot v_1
   \cdots v_t \ot 1) . 
\end{aligned}
\] 
Other verifications are tedious, but similar.
Details are given in the appendix. 
\end{proof}


\section{$\phi$-circle product formula and projections onto group components}\label{sec:circ-proj}

We first recall some basic facts about the  Hochschild cohomology of 
the skew group algebra $S(V)\# G$.
The graded vector space structure of the cohomology is well-known,
see for example \cite{farinati,ginzburgkaledin04,NPPT}.
We give some details here as will be needed for our bracket computations. 
We then  derive a formula for  $\phi$-circle products 
(defined in (\ref{eqn:phi-circle-def})) on this Hochschild cohomology, and define
projection operators needed for our main results.
We assume from now on that the characteristic of $k$ is 0.
Then 
\begin{equation}\label{eqn:G-invts}
  \HH^{\bu}(S(V)\# G)\cong \HH^{\bu}(S(V),S(V)\# G)^G,
\end{equation}
where the superscript $G$ denotes invariants of the action of $G$ on Hochschild
cohomology induced by its action on complexes (via the standard 
group action on tensor products and functions). 
This follows for example from 
\cite{gerst86}. 
We will focus our discussions and computations on $\HH^{\bu}(S(V),S(V)\# G)$,
and results for Hochschild cohomology $\HH^{\bu}(S(V)\# G)$ will follow by
restricting to its $G$-invariant subalgebra.

\subsection{Structure of the cohomology $\HH^{\bu}(S(V),S(V)\# G)$}
\label{cohomcomp}

Let $\{x_1,\dots, x_n\}$ be a basis for $V$ and  $\{x_1^*,\ldots,x_n^*\}$
the dual basis for its dual space $V^*$.  
The Hochschild cohomology $\HH^{\bu}(S(V),S(V)\# G)$ is computed as the homology of the complex
\begin{eqnarray*}
\Hom_{S(V)^e}(S(V)\ot \Wedge^{\bu}V\ot S(V) , S(V)\# G) & \cong & 
\Hom_k(\Wedge^{\bu}V, S(V)\# G)  \\ 
  & \cong &  
\bigoplus_{g\in G} S(V)\ox{\Wedge}^{\bu} V^\ast g.
\end{eqnarray*}
\par

The differential on the graded space $\bigoplus_{g\in G} S(V)\ox{\Wedge}^{\bu} V^\ast g$ induced by the above sequence of isomorphisms is left multiplication by the diagonal matrix
$$
E=\mathrm{diag}\{E_g\}_{g\in G}, \ \ \ E_g=\sum_{i}(x_i-{^gx_i})\partial_i,
$$
where $\partial_i = 1\ot x_i^*$.  So $E\cdot(\sum_g Y_gg)=\sum_g (E_g Y_g)g$.  This complex breaks up into a sum of subcomplexes $(S(V)\ox{\Wedge}^{\bu} V^\ast g,E_g)$, and we have
$$
\HH^{\bu}(S(V),S(V)\# G)=
\bigoplus_{g\in G} \coh^{\bu}(S(V)\ox{\Wedge}^{\bu} V^\ast g).
$$
We note that each $E_g$ is independent of the choice of basis, since it is simply the image of $1g \in \Hom_k(\Wedge^0V,S(V)g)$ under the differential on $\Hom_k(\Wedge^\bu V, S(V)g)$.
\par

The right $G$-action on  $\Hom_{S(V)^e}(K,S(V)\# G)$ is given by $(f\cdot g)(x)=g^{-1}f({^gx})g$, for $f\in\Hom_{S(V)^e}(K,S(V)\# G)$ and $g\in G$, giving it the structure of a $G$-complex.  This  translates to the action
\begin{equation}\label{eq.actn}
\big((a\ox f_1\dots f_l)g)\cdot h=\big(({^{h^{-1}}a})\ox f_1^h\dots f_l^h\big)h^{-1}gh
\end{equation}
on $\bigoplus_{g\in G} S(V)\ox{\Wedge}^{\bu} V^\ast g$, where $a\in S(V)$ and $f_i\in V^\ast$.  
It will be helpful to have the following general lemma.

\begin{lemma}\label{lem723}
Given any $G$-representation $M$, and element $g\in G$, there is a canonical complement to the $g$-invariant subspace $M^g$, given by $(M^g)^\perp=(1-g)\cdot M$.  This gives a splitting $M=M^g\oplus (1-g)M$ of $M$ as a $\langle g\rangle$-representation.  This decomposition satisfies
$$
M^g=M^{g^{-1}}\ \ \text{and}\ \ (1-g)M=(1-g^{-1})M
$$
and is compatible with the $G$-action in the sense that, for any $h\in G$,
$$
h\cdot M^g=M^{hgh^{-1}}\ \ \text{and}\ \ h\cdot\big((1-g)M\big)=(1-{hgh^{-1}})M.
$$
\end{lemma}

\begin{proof}
The operation $(1-g)\cdot-:M\to M$ has kernel precisely $M^g$.  Furthermore $M^g\cap (1-g)\cdot M=0$, since for any invariant element $m-gm$ in $(1-g)\cdot M$ we will have
$$
(m-gm)=\int_g(m-gm)=\int_gm-\int_g gm=\int_g m-\int_g m=0,
$$
where $\int_g=\frac{1}{|g|}\sum_{i=0}^{|g|-1}g^i$.  We conclude that $M=M^g\oplus (1-g)M$
when $M$ is finite dimensional, by a dimension count.  When $M$ is infinite dimensional, the result can be deduced from the fact that $M$ is the union of its finite dimensional submodules.  
\par

The equality $M^g=M^{g^{-1}}$ is clear, and the equality $(1-g)M=(1-g^{-1})M$ follows from the fact that $M=-gM$.  As for the compatibility claim, the identity $h\cdot M^g=M^{hgh^{-1}}$ is obvious, while the equality $h(1-g)M=(1-hgh^{-1})M$ follows from the computation
$$
h(1-g)M=h(1-g)h^{-1}M=(hh^{-1}-hgh^{-1})M=(1-hgh^{-1})M.
$$
\end{proof}

Let us take ${\det}_g^\perp$ to be the one dimensional $\langle g\rangle$-representation:
\[
{\rm {det}}_g^{\perp}=\Wedge^{\codim{V^g}}((1-g)V)^\ast g .
\]
We then have the embedding
\begin{equation}\label{bleh}
S(V^g)\ox{\Wedge}^{{\bu}-\codim V^g} (V^g)^\ast{\det}_g^\perp 
    \ \longrightarrow \  S(V)\ox{\Wedge}
^{\bu} V^\ast g
\end{equation}
induced by the embedding of $V^g$ as a subspace of $V$ and a corresponding
dual subspace embedding. 

\begin{lemma}\label{lem.654}
For each $g\in G$, 
$$
E_g\cdot(S(V^g)\ox{\Wedge}^{{\bu}-\codim V^g} (V^g)^\ast{\det}_g^\perp)=0
$$
and
$$
\left(\im(E_g\cdot-)\right)\cap \left(S(V^g)\ox{\Wedge}^{{\bu}-\codim V^g} (V^g)^\ast{\det}_g^\perp\right)=0.
$$
That is to say, the subspaces $S(V^g)\ox{\Wedge}^{{\bu}-\codim V^g} (V^g)^\ast{\det}_g^\perp$ consist entirely of cocycles and contain no nonzero coboundaries.
\end{lemma}

\begin{proof}
If we choose a basis $\{x_1,\dots, x_n\}$ for $V$ such that the first $l$ elements are a basis for $V^g$, and the remaining are a basis for $(1-g) V$, then we have
$$
E_g=\sum_{i=1}^n(x_i-{^gx_i})\partial_i=\sum_{i> l}(x_i-{^gx_i})\partial_i,
$$
since $x_i={^gx_i}$ for all $i\leq l$.  So $E_g\det^\perp_g=0$ and, if we let $\mathfrak{I}_g$ denote the ideal in $S(V)$ generated by $(1-g)V$, we have
$$
 E_g \left(S(V)\ox{\Wedge}^\bu V^\ast g\right)
 \subset \mathfrak{I}_g\ox {\Wedge}^\bu V^\ast g. 
$$
The second statement here implies that the proposed intersection is trivial.
\end{proof}

By the above information, there is an induced map
$$
S(V^g)\ox{\Wedge}^{{\bu}-\codim V^g} (V^g)^\ast{\det}_g^\perp 
  \ \longrightarrow \ \coh^{\bu}\big(S(V)\ox{\Wedge}^{\bu} V^\ast g\big)
$$
which is {\it injective}.  The following is a rephrasing of Farinati's calculation \cite{farinati}.

\begin{proposition}\label{prop_hhS(V)G}
The induced maps
$$
S(V^g)\ox{\Wedge}^{{\bu}-\codim V^g} (V^g)^\ast{\det}_g^\perp 
\ \longrightarrow \ \coh^{\bu}\big(S(V)\ox{\Wedge}^{\bu} V^\ast g\big)
$$
are isomorphisms for each $g\in G$, and so there is an isomorphism
\begin{equation}\label{751}
\bigoplus_{g\in G}\big(S(V^g)\ox{\Wedge}^{{\bu}-\codim V^g} (V^g)^\ast{\det}_g^\perp \big) \ \overset{\cong}\longrightarrow \  
\coh^{\bu}\big(\bigoplus_{g\in G}(S(V)\ox{\Wedge}^{\bu} V^\ast g)\big).
\end{equation}
\end{proposition}

Recalling that the codomain of (\ref{751}) is the  cohomology $\HH^{\bu}(S(V),S(V)\# G)$, the second portion of this proposition gives an identification
\begin{equation}\label{eqn:ds}
\bigoplus_{g\in G}\big(S(V^g)\ox{\Wedge}^{\bu-\codim V^g} (V^g)^\ast{\det}_g^\perp \big) = \HH^{\bu}(S(V),S(V)\# G).
\end{equation}
In addition to providing this description of the cohomology, the embedding (\ref{bleh}) is compatible with the $G$-action in the following sense.

\begin{proposition}\label{prop.Gsubcmplx}
\begin{enumerate}
\item
For any $g,h\in G$ there is an equality
$$
\big(S(V^g)\ox{\Wedge}^{{\bu}-\codim V^g} (V^g)^\ast{\det}_g^\perp\big)\cdot h=S(V^{h^{-1}gh})\ox{\Wedge}^{{\bu}-\codim V^{h^{-1}gh}} (V^{h^{-1}gh})^\ast{\det}_{h^{-1}gh}^\perp
$$
in $\bigoplus_{g} S(V)\ox{\Wedge}^{\bu} V^\ast g$.
\item The sum $\bigoplus_g S(V^g)\ox{\Wedge}^{{\bu}-\codim V^g} (V^g)^\ast{\det}_g^\perp$ is a $G$-subcomplex of the sum $\bigoplus_{g} S(V)\ox{\Wedge}^{\bu} V^\ast g$.
\item The isomorphism
$$
\bigoplus_{g\in G}\big(S(V^g)\ox{\Wedge}^{{\bu}-\codim V^g} (V^g)^\ast{\det}_g^\perp \big)\overset{\cong}\longrightarrow 
\coh^{\bu}\big(\bigoplus_{g\in G}(S(V)\ox{\Wedge}^{\bu} V^\ast g)\big)
$$
is one of graded $G$-modules.
\end{enumerate}
\end{proposition}

\begin{proof}
From Lemma \ref{lem723}, and the descriptions of $(V^g)^\ast$ and $((1-g)V)^\ast$ as those functions vanishing on $(1-g)V$ and $V^g$ respectively, we have
$$
(V^g)^*\cdot h= \{ \text{functions vanishing on }h^{-1}(1-g)V\}=(V^{h^{-1}gh})^*
$$
$$
((1-g)V)^*\cdot h= \{ \text{functions vanishing on }h^{-1}V^g\}=((1-h^{-1}gh)V)^* 
$$
and $h^{-1}\cdot S(V^g)=S(V^{h^{-1}gh})$.  This, along with the description (\ref{eq.actn}) gives the equality in (1).  Statement (2) follows from (1), and (3) follows from (2) and Proposition~\ref{prop_hhS(V)G}.
\end{proof} 

\subsection{The $\phi$-circle product formula}

We will compute first with the complex 
\[
  \Hom_{S(V)^e}(K,S(V)\# G)=\bigoplus_{g\in G}S(V)\ox{\Wedge}^{\bu} V^\ast g,
\]
and then restrict to $G$-invariant elements to make conclusions about
the cohomology $\HH^{\bu}(S(V)\# G)$.
Letting $\phi_K$ be the map 
of Definition~\ref{def:phi} and $\phi_{\widetilde{K}} = \phi_K\ot \id_{kG}$
as in equation (\ref{eqn:phiKtilde}), $\phi:= \phi_{\widetilde{K}}$ 
gives rise to a perfectly good {\it bilinear operation}
$$
Xg\circ_\phi Yh=\big(w\mapsto (-1)^{|w_1||Y|} Xg(\phi(w_1\ox Y(w_2)\ox {^hw_3}))
  h \big)  ,
$$
which need not be a chain map.  Here $X,Y\in S(V)\ox{\Wedge}^{\bu} V^\ast$, and $w\in K$.  We can also define the $\phi$-bracket in the most naive manner as 
$$
[Xg, Yh]_{\phi}=Xg\circ_{\phi} Yh- (-1)^{(|X|-1)(|Y|-1)}
   Yh\circ_{\phi} Xg.
$$
This operation, again, need not be well behaved at all on non-invariant functions, but it will be a bilinear map.
\par

In the lemma below, we 
give a formula for the $\phi$-circle product of special types of
elements. 
Our formula may be compared with \cite[Lemma~4.1]{Halbout-Tang10} and \cite[Theorem~7.2]{SW2}. 
Due to the structure of the Hochschild cohomology of $S(V)\# G$
stated in Proposition~\ref{prop_hhS(V)G}, this formula will in fact  
suffice to compute all brackets. 
To do this, we will only need to consider $G$-invariant elements
and compute relevant $\phi$-circle products on summands representing 
elements in $\HH^{\bu}(S(V),S(V)g)
\times \HH^{\bu}(S(V),S(V)h)$ for all pairs $g,h$ in $G$. 
This we will do in Section~\ref{sec:brackets}.

\begin{lemma}\label{lem:circles}
Let $g,h\in G$, $p\geq 0$, and 
$\omega_h$ any generator for ${\Wedge}^{\codim V^h}((1-h)V)^\ast$.  Consider any $X_i\in S(V)\ox V^\ast$, $\bar{Y}\in{\Wedge}^{p} V^\ast\omega_h$, and $Y=v_1\dots v_t\bar{Y}$ for $v_i\in V$.  Then 
\[
(X_1\dots X_d g\circ_\phi Yh)=\hspace{10.6cm}
\]

\vspace{-.6cm}

\[
\sum_{l,r,\sigma} 
\zeta_{r,|Y|}^{l,d,t}v_{\sigma(1)}\dots v_{\sigma(r-1)}X_l(v_{\sigma(r)}){^gv_{\sigma(r+1)}}\dots {^g v_{\sigma(t)}}(X_1\dots X_{l-1})\bar{Y}(X_{l+1}\dots X_d)gh
\]
where the sum is over $1\leq l\leq d$, $1\leq r\leq t$, $\sigma\in S_t$, and 
the coefficients $\zeta^{l,d,t}_{r,m}$ are the nonzero rational numbers
$$
\zeta^{l,d,t}_{r,m}=  (-1)^{(m-1)(l-1)} 
\frac{
(r+l-2)!(t-r+d-l)!}{(r-1)!(t-r)!(d+t+l-1)!}.
$$
\end{lemma}

\begin{proof}
Choose bases $y_1,\ldots, y_s$ of $(V^h)^{\perp}$ and $y_{s+1},\ldots, y_n$ of $V^h$.
Let $y_1^*,\ldots, y_n^*$ be the dual basis of $V^*$.  Recall the notation $\partial_j=1\ot y_j^\ast$.  We may write each of the $X_i$ as a sum of elements of the form $f\partial_l$, with $f\in S(V)$, and write $\bar{Y}$ as a sum of elements of the form $\partial_1\dots\partial_s \partial_J$, with $J$ an ordered subset of $\{{s+1},\dots ,n\}$.  By expanding both sides of the proposed equality accordingly, one sees that it suffices to prove the result in the case in which the $X_i$ are of the form $f_i\partial_{j_i}^\ast$ and $\bar{Y}=\partial_1\dots\partial_s \partial_J$.  We are free to reorder indices if necessary so that $\bar{Y}=\partial_1\dots\partial_m$ with $m\geq s$.
\par

Let us check the value of the $\phi$-circle product when applied to monomials of the form $o(y_{j_{l+1}},\ldots, y_{j_d},
   y_1,\ldots, y_m, y_{j_1},\ldots, y_{j_{l-1}})$, for arbitrary $l$.  If any of the $y_{j_i}$ are in $(V^h)^\perp$, or if any of the indices are repeated, then the given monomial will be $0$.  So we may assume that each of the $y_{j_i}$, for $i\neq l$, are in $V^h$ and that none of the indices are repeated.
\par

Applying the definition~\eqref{eqn:phi-circle-def}
of the $\phi$-circle product and the diagonal map (\ref{cmltp}), we find 
\[
\begin{aligned}
   & ((X_1\cdots X_d g)\circ_{\phi} (Yh)) \big( o(y_{j_{l+1}},\ldots, y_{j_d},
   y_1,\ldots, y_m, y_{j_1},\ldots, y_{j_{l-1}})\big) \hspace{1cm} \\
& \hspace{1cm}= (-1)^{(d-l)m} (X_1\cdots X_d g)\big(\phi (o(y_{j_{l+1}},\ldots, y_{j_d}) \ot
   v_1\cdots v_t \ot  {}^ho(y_{j_1},\ldots, y_{j_{l-1}})) h \big) .
\end{aligned}
\]
Since all the $y_{j_i}$ are in $V^h$, we may replace ${}^ho(y_{j_1},\ldots,y_{j_{l-1}})$ 
by $o(y_{j_1},\ldots, y_{j_{l-1}})$ in the above expression.  Therefore the value of $(X_1\cdots X_d g)\circ_{\phi} (Yh)$ on
the given monomial is as follows, using
Definition~\ref{def:phi} of $\phi$: 
\[
\begin{aligned}
 & (-1)^{(d-l)m}(X_1\cdots X_d g)\big(\phi(o(y_{j_{l+1}},\ldots, y_{j_d})\ot
   v_1\cdots v_t \ot o(y_{j_1},\ldots,y_{j_{l-1}}) h\big) \hspace{1.5cm}\\
& \hspace{.1cm}=(-1)^{(d-l)m} (X_1\cdots X_d g) \big( 
    \sum_{\substack{1\leq r\leq t\\ \sigma\in S_t}}
   c_r^{d-l, t,l-1} 
  v_{\sigma(1)}\cdots v_{\sigma(r-1)} \ot  \\
&\hspace{4.5cm}   o(y_{j_1},\ldots,y_{j_{l-1}}, v_{\sigma(r)}, y_{j_{l+1}},\ldots,y_{j_d})
  \ot v_{\sigma(r+1)}\cdots v_{\sigma(t)} h \big)\\
& \hspace{.1cm}= (-1)^{(d-l)m } \! \sum_{\substack{1\leq r\leq t\\
   \sigma\in S_t}} c_r^{d-l, t, l-1} 
  v_{\sigma(1)}\cdots v_{\sigma(r-1)}  
    X_l(v_{\sigma(r)})  f_1\cdots f_{l-1} f_{l+1} \cdots f_d g
     v_{\sigma(r+1)}\cdots v_{\sigma(t)} h .
\end{aligned}
\] 
The above argument can also be used to show that the circle product vanishes on all monomials $o(y_I)$ which are not of the form $\pm o(y_{j_1},\ldots,y_{j_{l-1}},y_1,\ldots,y_m,y_{j_{l+1}},\ldots,y_{j_d})$ for some $l$.  To obtain the equality of the theorem, we must reorder the factors in our
initial argument: 
\[
\begin{aligned}
&  o(y_{j_{l+1}},\ldots, y_{j_d},y_1,\ldots, y_m,y_{j_1},\ldots, y_{j_{l-1}}) \hspace{1cm}\\
& \hspace{1cm} = 
(-1)^{m(l-1) + (d-l)(l-1)+ (d-l)m} o(y_{j_1},\ldots,y_{j_{l-1}},y_1,\ldots,y_m,
y_{j_{l+1}},\ldots, y_{j_{d}}) . 
\end{aligned}
\]
Multiply by this coefficient and compare values with those of the function in
the statement of the theorem to see that they are the same. 
\end{proof}

\subsection{Projections onto group components} 
For each $g\in G$, we will construct a chain retraction
$$
p_g: S(V)\ox{\Wedge}^{\bu} V^\ast g \ \longrightarrow \ S(V^g)\ox{\Wedge}^{\bu-\codim V^g} (V^g)^\ast{\det}_g^\perp
$$
onto the subspace $S(V^g)\ox{\Wedge}^{\bu-\codim V^g} (V^g)^\ast{\det}_g^\perp$.  (The differential on the codomain is taken to be $0$.)  Simply by virtue of being a retract of an injective quasi-isomorphism, each $p_g$ will also be a quasi-isomorphism.  
\par

In the sections that follow we will often think of the $p_g$ as quasi-isomorphisms from $S(V)\ox{\Wedge}^{\bu} V^\ast g$ to itself, simply by composing with the embedding.  We outline the construction of $p_g$ below.

\begin{construction_pg} From the canonical decomposition $V=V^g\oplus (1-g)V$ we get an identification $S(V)=S(V^g\oplus (1-g)V)$ and canonical projection $p^1_g:S(V)\to S(V^g)$.  We also get a canonical decomposition of the dual space and its higher wedge powers,
$$
V^\ast=(V^g)^\ast\oplus ((1-g)V)^\ast\ \ \mathrm{and}\ \ {\Wedge}^i V^\ast=\bigoplus_{i_1+i_2=i}({\Wedge}^{i_1}(V^g)^\ast)\wedge({\Wedge}^{i_2}((1-g)V)^\ast),
$$
whence we have a second canonical projection
$$
p^2_g:{\Wedge}^{\bu} V^\ast\to {\Wedge}^{\bu-\codim V^g} (V^g)^\ast\wedge ({\Wedge}^{\codim V^g}((1-g)V)^\ast)={\Wedge}^{\bu-\codim V^g} (V^g)^\ast{\det}_g^\perp .
$$
We now define $p_g$ as the tensor product $p_g^1\ox p_g^2$,
\begin{equation}\label{773}
p_g:S(V)\ox {\Wedge}^{\bu} V^\ast g \ \longrightarrow \ 
 S(V^g)\ox {\Wedge}^{\bu-\codim V^g}(V^g)^\ast{\det}^\perp_g  . 
\end{equation}
\end{construction_pg}

It is apparent from the construction that each $p_g$ restricts to the identity on the subspace $S(V^g)\ox \Wedge^{\bu-\codim V^g}(V^g)^\ast{\det}_g^\perp$.  Furthermore, since the ideal $\mathfrak{I}_g$ generated by $(1-g)V$ is precisely the kernel of $p_g^1$, and left multiplication by $E_g$ has image in $\mathfrak{I}_g\ox{\Wedge}^\bu V^\ast$, we see that $p_g(E_g\cdot-)=0$.  This is exactly the statement that $p_g$ is a chain map.  Note that, by Proposition \ref{prop_hhS(V)G}, the projections will be quasi-isomorphisms.
\par

Recall that, by Proposition \ref{prop.Gsubcmplx}, the subspace
$$
\bigoplus_g S(V^g)\ox{\Wedge}^{{\bu}-\codim V^g} (V^g)^\ast{\det}_g^\perp
 \ \subset \ \bigoplus_{g} S(V)\ox{\Wedge}^{\bu} V^\ast g
$$
is a $G$-subcomplex.  These projections $p_g$ are compatible with the $G$-action in the sense of 

\begin{lemma}\label{prop_pginv}
\begin{enumerate}
\item For any $Xg\in S(V)\ox {\Wedge}^{\bu} V^\ast g$ the projections $p_g$ and $p_{h^{-1}gh}$ satisfy the relation
$$
p_{h^{-1}gh}((Xg)\cdot h\big)= p_{g}(X_g)\cdot h.
$$
\item The coproduct map
$$
p:\bigoplus_{g} S(V)\ox{\Wedge}^{\bu} V^\ast g \ \longrightarrow \ \bigoplus_g S(V^g)\ox{\Wedge}^{{\bu}-\codim V^g} (V^g)^\ast{\det}_g^\perp,
$$
i.e. the map $p=\mathrm{diag}\{p_g:g\in G\}$, is a $G$-linear quasi-isomorphism.
\item If a sum of elements $\sum_g X_g g$ is $G$-invariant then so is $\sum_g p_g(X_gg)$. 
\end{enumerate}
\end{lemma}

\begin{proof}
As was the case in the proof of Proposition \ref{prop.Gsubcmplx}, (1) follows from the compatibilities of the decompositions $V=V^g\oplus (1-g)V$ with the $G$-action given in Lemma \ref{lem723}.  Statements (2) and (3) follow from (1) and the fact that each $p_g$ is a quasi-isomorphism.
\end{proof}

In the following results, we take 
$\mathfrak{I}_g\subset S(V)$ to be the ideal generated by $(1-g)V$, as was done above.  
Given an ordered subset $I=\{i_1,\ldots, i_j\}$ of 
$\{1,\ldots,n\}$, let $\partial_I=\partial_{i_1}\cdots\partial_{i_j}$
in $S(V)\ot \Wedge^{\bu}V^*$ where $\partial_i = 1\ot x_i^*$ as before. 

\begin{lemma}\label{prop789}
Let $\omega_g$ be an arbitrary generator for ${\Wedge}^{\codim V^g}((1-g)V)^\ast$.  The kernel of the projection $p_g$ defined in (\ref{773}) is the sum
$$
\ker(p_g)=\mathfrak{I}_g\ox {\Wedge}^{\bu} V^\ast g+S(V)\cdot\{\partial_Ig:\omega_g\mathrm{\ does\ not\ divide\ }\partial_I\}.
$$
\end{lemma}

If we choose bases $\{x_1,\ldots, x_l\}$ of $V^g$ and $\{x_{l+1},\dots, x_n\}$ of $(1-g)V$, and take $\omega_g=\partial_{l+1}\cdots\partial_n$, the second set can be written as
$$
S(V)\cdot\{\partial_Ig:\{l+1,\dots, n\}\mathrm{\ is\ not\ a\ subset\ of\ }I\} .
$$

\begin{proof}
Note that, for $p_g^1$ and $p_g^2$ as in the above construction of $p_g$, we have
$$
\ker(p_g^1)=\mathfrak{I}_g\ \ \text{and}\ \ \ker(p_g^2)=k\{\partial_Ig:\omega_g\mathrm{\ does\ not\ divide\ }\partial_I\}
$$
So the description of $\ker(p_g)$ follows from the fact that for any product of vector space maps $\sigma_1\ox\sigma_2:W_1\ox W_2\to U_1\ox U_2$, its kernel is the sum $\ker(\sigma_1)\ox W_2+W_1\ox\ker(\sigma_2)$.
\end{proof}

\section{Brackets}\label{sec:brackets} 

In this section we assume the characteristic of $k$ is 0, and 
derive a general formula for brackets on Hochschild
cohomology of $S(V)\# G$, using the $\phi$-circle product formula of 
Lemma~\ref{lem:circles} and the projection maps (\ref{773}). 
The Schouten bracket for the underlying symmetric algebra 
features prominently.
We use our formula to obtain several conclusions about brackets, in
particular some vanishing criteria.

We will use the notation and results of Section~\ref{sec:circ-proj}.
In particular, we will express elements of the Hochschild cohomology
$\HH^{\bu}(S(V)\# G)$ as $G$-invariant elements of $\HH^{\bu}(S(V),
S(V)\# G)$ by (\ref{eqn:G-invts}), and we will use the 
identification of cohomology, 
$$
\bigoplus_{g\in G} S(V^g)\ox{\Wedge}^{{\bu}-\codim V^g} (V^g)^\ast {\rm {det}}_g^\perp  =  
\HH^{\bu}(S(V),S(V)\# G)
$$
given by (\ref{eqn:ds}). 
Elements of cohomology $\HH^{\bu}(S(V),S(V)\# G)$ that are nonzero only in the
component indexed by a unique $g$ in $G$ in the above sum 
are said to be {\em supported on $g$}.
Elements  that are nonzero only in 
components indexed by elements in the conjugacy class of  $g$  are
said to be {\em supported on the conjugacy class of $g$}. 
The canonical projections
$
p_g:S(V)\ox{\Wedge}^{\bu} V^\ast g\to S(V^g)\ox{\Wedge}^{\bu} (V^g)^\ast 
{\rm {det}}_g^\perp g 
$,
defined in (\ref{773}), 
will appear in our expressions of  brackets on Hochschild cohomology $\HH^{\bu}(S(V)\# G)$ below.
\par

\subsection{Preliminary information on the Schouten bracket and group actions}

We let the bracket $\{X,Y\}$ denote the standard Schouten bracket on $S(V)\ox {\Wedge}^{\bu} V^\ast\cong{\Wedge}^{\bu}_{S(V)}T$, where $T$ denotes the global vector fields on $\Spec(S(V))$ (or rather, algebra derivations on $S(V)$).  Gerstenhaber brackets in this case are precisely Schouten brackets,
and we briefly verify that our approach does indeed give this expected result: 

\begin{lemma}\label{lem908}
For any $X,Y\in S(V)\ox{\Wedge}^{\bu} V^\ast$,
$$
[X,Y]_\phi=\{X,Y\} .
$$
\end{lemma}
\begin{proof}
By construction, the restriction of $\phi$ to $K\ox_{S(V)}K\subset \tilde{K}\ox_{S(V)\# G}\tilde{K}$ provides a contracting homotopy for $F_K$.  Therefore, for elements in $S(V)\ox{\Wedge}^\bu V^\ast$, we will have $[X,Y]_\phi=[X,Y]_{\phi|K\ox_{S(V)} K}$.  So it does make sense to consider the proposed equality $[X,Y]_\phi=\{X,Y\}$ in $S(V)\ox{\Wedge}^{\bu} V^\ast\cong {\Wedge}^\bu_{S(V)}T$.
\par

It is shown in \cite[\S4]{Negron-Witherspoon} that some choice of contracting homotopy $\psi$ is such that $[X,Y]_\psi$ is the Schouten bracket.  By~\cite[Theorem~3.2.5]{Negron-Witherspoon}, for any two choices of contracting homotopy, the difference in the associated brackets is a coboundary.  Since the differential vanishes on $\Hom_{A^e}(K,S(V))=S(V)\ox{\Wedge}^{\bu} V^\ast$, we see that
$[X,Y]_\phi=[X,Y]_\psi=\{X,Y\}$.
\end{proof}

Under the identification $S(V)\ox{\Wedge}^\bu V^\ast\cong {\Wedge}^\bu_{S(V)}T$, the right action of an element $g\in G$ on $S(V)\ox V^\ast$ is identified with conjugation by the corresponding automorphism, $X\cdot g=X^g=g^{-1}Xg$.  On higher degree elements the action is given by the standard formula, $(X_1\cdots X_l)\cdot g=(X_1\cdots X_l)^g=X_1^g\cdots X_l^g$.  Under the identification 
$$
\bigoplus_{g\in G} (S(V)\ox{\Wedge}^{\bu} V^\ast g)\cong \bigoplus_{g\in G} ({\Wedge}^{\bu}_{S(V)} T g) , 
$$
the action is given by $(Xg)\cdot h=X^h{h^{-1}gh}$, where we can view $X$ either as an element in $S(V)\ox V^\ast$ or as a polyvector field.

\begin{lemma}
For any $X,Y\in S(V)\ox{\Wedge}^{\bu} V^\ast$, and $g\in G$, we have
$\{X,Y\}^g=\{X^g,Y^g\}$.
\end{lemma}
\begin{proof}
We identify $S(V)\ox{\Wedge}^{\bu} V^\ast$ with the global polyvector fields ${\Wedge}^{\bu}_{S(V)}T$.  Then the lemma follows from the fact that the $G$-action is simply given by conjugating by the corresponding automorphism, the fact that the Schouten bracket is given by composition of vector fields on $T$, and the Gerstenhaber identity $\{X, Y_1Y_2\}=\{X,Y_1\}Y_2\pm Y_1\{X,Y_2\}$.
\end{proof}

\begin{lemma}
For any $G$-invariant elements $\sum_g X_gg$ and $\sum_{h\in G} Y_hh$ in the sum $\bigoplus_{g\in G} (S(V)\ox{\Wedge}^{\bu} V^\ast g)$, the element $\sum_{g,h\in G}\{X_g,Y_h\}gh$ is also $G$-invariant.  Furthermore, for any such $\sum_g X_gg$ and $\sum_h Y_hh$, the element $\sum_{g,h\in G}p_{gh}\{X_g,Y_h\}gh$ will be a $G$-invariant cocycle.
\end{lemma}

\begin{proof}
We know an element $\sum_g Z_g g$ will be invariant if and only if, for each $g,\sigma\in G$, $Z_g^\sigma=Z_{\sigma^{-1}g\sigma}$.  So the $X_g$ and $Y_h$ have this property, and it follows that the sum $\sum_{\{g,h\in G: gh=\tau\}}\{X_g,Y_h\}$ will have this property for each $\tau\in G$ since
$$
\sum_{gh=\tau}\{X_g,Y_h\}^\sigma=\sum_{gh=\tau}\{X^\sigma_g,Y^\sigma_h\}=\sum_{gh=\tau}\{X_{\sigma^{-1}g\sigma},Y_{\sigma^{-1}h\sigma}\}=\sum_{\{g',h':g'h'=\sigma\tau\sigma^{-1}\}}\{X_{g'},Y_{h'}\}.
$$
The last statement 
now follows directly from Lemma \ref{prop_pginv}(3).
\end{proof}

\subsection{$\phi$-brackets as Schouten brackets}

Before we begin, it will be useful to have the following two lemmas.  Recall that $\mathfrak{I}_g$ is the ideal generated by $(1-g)V$ in $S(V)$.

\begin{lemma}\label{lem:Ig}
For any vectors $u_{i_1},\dots, u_{i_\nu}\in V$ and $g\in G$, the element ${^{(1-g)}(u_{i_1}\dots u_{i_\nu})}$ is in $\mathfrak{I}_g$.
\end{lemma}

\begin{proof}
We proceed by induction on the number of vectors $\nu$.  When $\nu=1$ the result is immediate.  For $\nu>1$ we have
$$
\ba{rl}
{^{(1-g)}(u_{i_1}\dots u_{i_\nu})}&=u_{i_1}\dots u_{i_\nu}-{^gu_{i_1}}\dots {^gu_{i_\nu}}\\
&=(u_{i_1}\dots u_{i_{\nu-1}}-{^gu_{i_1}}\dots {^gu_{i_{\nu-1}}})u_{i_\nu}+{^gu_{i_1}}\dots {^gu_{i_{\nu-1}}}{^{(1-g)}u_{i_\nu}},
\ea
$$
which is now in $\mathfrak{I}_g$ by induction.
\end{proof}

\begin{lemma}\label{lem954}
Suppose $c$ and $c'$ are invariant cocycles in $\bigoplus_g S(V)\ox{\Wedge}^{\bu}V^\ast g$ that differ by a (possibly non-invariant) coboundary.  Then $c$ and $c'$ differ by an invariant coboundary.
\end{lemma}

\begin{proof}
Note that $c-c'$ is also invariant.  Let $b$ be any element with $d(b)=c-c'$, where $d$ is the differential $d=E\cdot-$.  Then we have
$$
c-c'=d(b)\cdot\int_G=d(b\cdot\int_G),
$$
where $\int_G$ is the standard integral $\frac{1}{|G|}\sum_{g\in G}g$.  So $b\cdot\int_G$ provides the desired invariant bounding element.
\end{proof}

We can now give a general formula for the Gerstenhaber bracket on
$\HH^{\bu}(S(V)\# G)$  in terms of Schouten brackets.
One may compare with~\cite[Theorem~4.4, Corollary~4.11]{Halbout-Tang10} where the authors
give similar formulas under some conditions on the group $G$ and its action on $V$.

\begin{theorem}\label{thm:gh}
Let $X=\sum_g X_g g$ and $Y=\sum_h Y_h h$ be $G$-invariant cocycles in
$\oplus_{g\in G} S(V^g)\ot \Wedge^{\bu -\codim V^g} (V^g)^* \det^{\perp}_g$.
The sum $\sum_{g,h\in G} p_{gh} \{X_g,Y_h\} gh$ is a $G$-invariant
cocycle and,
considered as elements of  the cohomology $\HH^\bu(S(V)\# G)$,
$$ 
  [X,Y] = 
 \sum_{g,h\in G} p_{gh} \{ X_g, Y_h\} gh.
$$
\end{theorem}

\begin{proof}
By Lemma~\ref{lem954}, it suffices to show that the equality holds up to arbitrary coboundaries.  Note that 
$$
p[X,Y]_\phi=\sum_{g,h} p_{gh}(X_gg\circ_\phi Y_hh)-(-1)^{(|Y|-1)(|X|-1)}\sum_{g,h} p_{hg}(Y_hh\circ_\phi X_gg).
$$
By considering the group automorphism
$$
G\times G\to G\times G,\ \ (g,h)\mapsto (g,ghg^{-1}),
$$
and the equality $ghg^{-1}g=gh$, we see that we can reindex the second sum to obtain 
\begin{equation}\label{eq:1219}
\ba{rl}
p[X,Y]_\phi&=\sum_{g,h} p_{gh}(X_gg\circ_\phi Y_hh)\mp\sum_{g,h} p_{gh}(Y_{ghg^{-1}}ghg^{-1}\circ_\phi X_gg)\\
&=\sum_{g,h} \Big(p_{gh}(X_gg\circ_\phi Y_hh)\mp p_{gh}(Y_{ghg^{-1}}ghg^{-1}\circ_\phi X_gg)\Big).
\ea
\end{equation}
We claim that there is an equality
\begin{equation}\label{eq:1226}
p_{gh}(X_gg\circ_\phi Y_hh)\mp p_{gh}(Y_{ghg^{-1}}ghg^{-1}\circ_\phi X_gg)=p_{gh}[X_g,Y_h]_\phi gh=p_{gh}\{X_g,Y_h\}gh
\end{equation}
for each pair $g,h\in G$, where the second equality follows already by Lemma \ref{lem908}.  If we can establish (\ref{eq:1226}) then we are done, by the final expression in (\ref{eq:1219}) and the fact that the difference $[X,Y]_\phi-p[X,Y]_\phi$ is a coboundary.
\par

Let us fix elements $g,h\in G$.  Write $X_g$ as a sum of elements of the form $u_1\cdots u_s \bar{X}$ with the $u_j\in V^g$, and $Y_h$ as a sum of elements $v_1\cdots v_t \bar{Y}$ with the $v_i\in V^h$.  Here $\bar{X},\bar{Y}\in \Wedge^\bullet V^\ast$.  By $h$-invariance there is an equality ${}^gv_i={}^{gh}v_i$ for each $i$.  For arbitrary $Z$ in $S(V)\ot \Wedge^{\bu} V^\ast$ and elements $a_i\in S(V)$, the projection  $p_{gh}({^{(1-gh)}(a_1\dots a_q)}Zgh)$ vanishes by Lemma~\ref{prop789} and Lemma~\ref{lem:Ig}.  So for any $r<t$ and $\sigma\in S_n$ there will be equalities
\begin{equation}\label{eq:1198}
\ba{rl}
p_{gh}\big({^g(v_{\sigma(r+1)}\dots v_{\sigma(t)})}Zgh\big)&=p_{gh}\big({^{gh}(v_{\sigma(r+1)}\dots v_{\sigma(t)})}Zgh\big)\\
&=p_{gh}(v_{\sigma(r+1)}\dots v_{\sigma(t)}Zgh).
\ea
\end{equation}
It now follows from the expression for the circle operation given in Lemma \ref{lem:circles} that there is an equality
\begin{equation}\label{eq:1237}
p_{gh}(X_gg\circ_\phi Y_hh)=p_{gh}(X_g\circ_\phi Y_h) gh.
\end{equation}
This covers half of what we need.
\par

We would like to show now
\[
p_{gh}(Y_{ghg^{-1}}ghg^{-1}\circ_\phi X_gg)=p_{gh}(Y_h\circ_\phi X_g) .
\]
Simply replacing $g$ and $h$ with $ghg^{-1}$ and $g$ in (\ref{eq:1237}), as well as $X_g$ with $Y_{ghg^{-1}}$ and $Y_h$ with $X_g$, gives
\[
p_{gh}(Y_{ghg^{-1}}ghg^{-1}\circ_\phi X_gg)=p_{gh}(Y_{ghg^{-1}}\circ_\phi X_g) gh.
\]
Now $G$-invariance of $Y$ implies immediately $Y_{ghg^{-1}}=Y^{g^{-1}}_h$.  So we have
\[
p_{gh}(Y_{ghg^{-1}}ghg^{-1}\circ_\phi X_gg)=p_{gh}(Y^{g^{-1}}_h\circ_\phi X_g) gh ,
\]
whence we need to show $p_{gh}(Y^{g^{-1}}_h\circ_\phi X_g) gh=p_{gh}(Y_h\circ_\phi X_g) gh$.
\par

Recall our expressions for $Y_h$ and $X_g$ from above, in terms of the $v_i$, $u_j$, $\bar{Y}$ and $\bar{X}$.  In the notation of Lemma~\ref{prop789}, we may assume that $\omega_g|\bar{X}$, and hence that $\omega_{g^{-1}}|\bar{X}$ by Lemma~\ref{lem723}.  We write $\bar{Y}$ as a sum of monomials $Y_1\cdots Y_e$ for functions $Y_i\in V^\ast$.  For each $Y_i$ and $u_j$ we have $Y_i^{g^{-1}}(u_j)=Y_i({^{g^{-1}}u_j})=Y_i(u_j)$.  We also have $Y_i^{g^{-1}}\bar{X}=Y_i\bar{X}$ since $\omega_{g^{-1}}|\bar{X}$.  From these observations and the expression of Lemma~\ref{lem:circles}, since $Y_h$ is a sum of elements of the form $v_1\cdots v_t \overline{Y}$, we deduce an equality
\[
Y_h^{g^{-1}}\circ_\phi X_g=\sum {^g(v_1\dots v_t)}(\bar{Y}^{g^{-1}})\circ_\phi X_g=\sum {^g(v_1\dots v_t)}\bar{Y}\circ_\phi X_g.
\]
Finally, by the same argument given for the equality (\ref{eq:1198}) we find also that
\[
\sum p_{gh}\big({^g(v_1\dots v_t)}\bar{Y}\circ_\phi X_g\big)=\sum p_{gh}\big((v_1\dots v_t)\bar{Y}\circ_\phi X_g\big)=p_{gh}(Y_h\circ_\phi X_g).
\]
Taking these two sequences of equalities together gives the desired equality
\[
p_{gh}(Y_h^{g^{-1}}\circ_\phi X_g)=p_{gh}(Y_h\circ_\phi X_g),
\]
establishes (\ref{eq:1226}), and completes the proof.
\end{proof}

\subsection{Corollaries: The codimension grading and some general vanishing results}

We apply the formula in Theorem \ref{thm:gh} to analyze distinct cases.  In one case, we  consider brackets with an element $X$ supported on group elements that act trivially on $V$.  In another case, we  consider brackets of $X$ and $Y$ supported on elements that act nontrivially.  We will have in this second case some general vanishing results.  The following  observation
helps in organizing these cases. 

\begin{observation}\label{obs:observer}
The graded $G$-module
\begin{equation}\label{eq:1248}
\bigoplus_{g\in G}S(V^g)\ox{\Wedge}^{{\bu}-\codim V^g} (V^g)^\ast {\det}^\perp_g
\end{equation}
decomposes as a direct sum of graded $G$-subspaces $\mathscr{D}(i)$, for $0\leq i\leq \dim(V)$, as does its $G$-invariant subspace,
\begin{equation}\label{eq:codimdecomp}
\left(\bigoplus_{g\in G}S(V^g)\ox{\Wedge}^{{\bu}-\codim V^g} (V^g)^\ast {\det}^\perp_g\right)^G= \ \bigoplus_{0\leq i\leq \dim(V)} \mathscr{D}(i)^G.
\end{equation}
The subspace $\mathscr{D}(i)$ consists of all sums of elements supported on group elements $g$ for which $\mathrm{codim}V^g=i$.  We call the decomposition (\ref{eq:codimdecomp}) the codimension grading for $\HH^\bu(S(V)\# G)$.
\end{observation}

Said another way, $\mathscr{D}(i)$ consists of all summands in (\ref{eq:1248}) whose first nonzero cohomology class occurs in degree $i$.  Note that classes in $\mathscr{D}(0)$ are supported on only those group elements which act trivially on $V$.  The brackets between elements in $\mathscr{D}(0)^G$ will just be given by the Schouten brackets (cf.\ \cite[Corollary 7.4]{SW2}): 

\begin{corollary}\label{cor:brackform0}
Let $X=\sum_g X_gg$ and $Y=\sum_h Y_hh$ be $G$-invariant cocycles in $\oplus_{g\in G}S(V^g)\ox{\Wedge}^{{\bu}-\codim V^g} (V^g)^\ast {\det}^\perp_g$ that are supported on group elements acting trivially, i.e.\ $X,Y\in \mathscr{D}(0)$.  Then in cohomology,
$$
[X,Y]=\sum_{g,h} \{X_g,Y_h\}gh.
$$ 
\end{corollary}

\begin{proof}
In this case for each $g,h$ with $X_g$ and $Y_h$ nonzero we will have $V^g=V^h=V^{gh}=V$ and $p_g=id$, $p_h=id$ and $p_{gh}=id$.
The result now follows from Theorem~\ref{thm:gh}.
\end{proof}

We refer directly to Theorem \ref{thm:gh} for information on the bracket between cochains in $\mathscr{D}(0)^G$ and $\mathscr{D}(>0)^G$.  We next give some conditions under which brackets are 0.  The following corollary was first proved in \cite{SW2} using different techniques.

\begin{corollary}[{\cite[Proposition 8.4]{SW2}}]\label{cor:vanish1}
Let $g,h\in G$ be such that $(V^g)^{\perp}\cap (V^h)^{\perp}$ is nonzero and
is a $kG$-submodule of $V$. 
Let $X,Y$ be $G$-invariant cocycles in the sum $\oplus_{g\in G} 
S(V^g)\ot \Wedge^{\bu -\codim V^g} (V^g)^* \det^{\perp}_g$ supported on
the conjugacy classes of $g,h$, respectively. Then 
\[
[X,Y] = 0.
\]
\end{corollary}

\begin{proof}
The hypotheses imply that $(V^{aga^{-1}})^{\perp}\cap (V^{bhb^{-1}})^{\perp}$
is nonzero for all $a,b\in G$.
We will argue that $X_g g\circ_{\phi} Y_h h =0$ at the chain level, and similar
reasoning will apply to $X_{aga^{-1}} aga^{-1}\circ_{\phi} Y_{bhb^{-1}}bhb^{-1}$
and to $Y_{bhb^{-1}}bhb^{-1} \circ_{\phi} X_{aga^{-1}} aga^{-1}$.
Consider the argument $o(y_{j_1},\ldots,y_{j_{l-1}},y_1,\ldots, y_m, y_{j_{l+1}},\ldots, y_{j_d})$
in the proof of Lemma~\ref{lem:circles}. This can be nonzero only
in case  $y_{j_l}\in (V^g)^{\perp}\cap (V^h)^{\perp}$, due to linear 
dependence of the vectors involved otherwise.  
Thus the only possible terms
in the $\phi$-circle product formula of Lemma~\ref{lem:circles} 
that could be nonzero are indexed by such $l$.
However then $X_l(v_{\sigma(r)})=0$ for
all $r,\sigma$, since $v_{\sigma(r)}\in V^h$, in the notation of the proof
of Lemma~\ref{lem:circles}. 
\end{proof} 

The following corollary was pointed out to us by Travis Schedler.

\begin{corollary}\label{cor:1315}
For classes $X\in\mathscr{D}(i)^G$ and $Y\in\mathscr{D}(j)^G$ we have $[X,Y]\in\mathscr{D}(i+j)^G$.
\end{corollary}

\begin{proof}
As was argued in the proof of Corollary~\ref{cor:vanish1} we find that $X'\circ_\phi Y'=0$ whenever 
\[
X'\in S(V^g)\ox{\Wedge}^{{\bu}-\codim V^g} (V^g)^\ast {\det}^\perp_g\ \ \mathrm{and}\ \ 
Y'\in S(V^h)\ox{\Wedge}^{{\bu}-\codim V^h} (V^h)^\ast {\det}^\perp_h
\]
and $(V^g)^\perp\cap (V^h)^\perp\neq 0$.  That is to say, nonvanishing of the element $X'\circ_\phi Y'$ implies $(V^g)^\perp\cap (V^h)^\perp=0$ and therefore $\mathrm{codim}V^{gh}=\mathrm{codim}V^g+\mathrm{codim}V^h$~\cite[Lemma 2.1]{SW}.  It follows that for (even non-invariant) classes $X=\sum_g X_gg\in \mathscr{D}(i)$ and $Y=\sum_h Y_hh\in\mathscr{D}(j)$ we have
\[
[X,Y]_\phi=\sum_{g,h} [X_gg,Y_hh]_\phi\in \mathscr{D}(i+j).
\]
\end{proof}

\begin{corollary}\label{cor:codimgrading}
The Hochschild cohomology $\HH^\bu(S(V)\# G)$ is a graded Gerstenhaber algebra with respect to the codimension grading.
\end{corollary}

\begin{proof}
We just saw in Corollary~\ref{cor:1315} that the Gerstenhaber bracket respects the codimension grading.  It has also already been established that the cup product respects the codimension grading~\cite{etingofginzburg}, \cite[Proposition 8.1, Theorem 9.2]{SW}.
\end{proof}

The following corollary generalizes \cite[Theorem~9.2]{SW2}, where it
was proven in homological degree~2.
The cocycles $X,Y$ in the corollary are by hypothesis of smallest possible homological
degree in their group components.

\begin{corollary}\label{thm:supp-off-id}
Take any $G$-invariant cocycles $X= \sum_{g\in G} X_g g$,
$Y=\sum_{h\in G} Y_h h$ in $\oplus_{g\in G} S(V^g)\ot \det_g^{\perp} $
that are supported on elements that act nontrivially on $V$.
Then $[X,Y]= 0$ on cohomology. 
\end{corollary}

\begin{proof}
We have that $X$ and $Y$ are sums of invariant classes $X'$ and $Y'$ which are of minimal degrees in $\mathscr{D}(i)$ and $\mathscr{D}(j)$ respectively, i.e. degrees $i$ and $j$.  By Corollary~\ref{cor:1315} we have $[X',Y']\in\mathscr{D}(i+j)$.  But now since the Gerstenhaber bracket is homogeneous of degree $-1$ we have that $[X',Y']$ is of degree $i+j-1$.  Since there are no nonzero classes of degree $i+j-1$ in $\mathscr{D}(i+j)$ we have that all $[X',Y']=0$ and hence $[X,Y]=0$.
\end{proof}





\section{(Non)vanishing of brackets in the case $(V^g)^\perp\cap(V^h)^\perp=0$}
\label{sec:nonvanishing} 

With Corollaries \ref{cor:vanish1} and \ref{thm:supp-off-id} we seem to be approaching a general result.  Namely, that for any $X$ and $Y$ supported on group elements that act nontrivially we will have $[X,Y]=0$.  It is even known that such a vanishing result holds in degree~$2$ by \cite[Theorem 9.2]{SW2}.  This is, however, not going to be the case in higher degrees.  The result even fails to hold when we consider the bracket $[X,Y]$ of elements in degrees~$2$ and $3$ (see Example~\ref{ex:two} below).  We give in this section a few examples to illustrate this nonvanishing, and (re)establish vanishing in degree $2$.

\subsection{Some examples for which $(V^g)^\perp\cap(V^h)^\perp=0$}

The following two examples illustrate, first, the essential role of taking invariants in establishing the degree $2$ vanishing result of \cite[Theorem 9.2]{SW2} and, second, an obstruction to establishing a general vanishing result in the case $(V^g)^\perp\cap(V^h)^\perp=0$.

\begin{example}\label{ex:one}
Take $G=\mathbb{Z}/2\mathbb{Z}\times\mathbb{Z}/2\mathbb{Z}$ and $V=k\{x_1,x_2,x_3\}$.  Let $g$ and $h$ be the generators of the first and second copies of $\mathbb{Z}/2\mathbb{Z}$, and take the $G$-action on $V$ defined by
$$
g\cdot x_i=(-1)^{\delta_{i1}}x_i\ \ \text{and}\ \ h\cdot x_j=(-1)^{\delta_{j3}}x_j.
$$
So $V^g=k\{x_2,x_3\}$, $(1-g)V=k x_1$, $V^h=k\{x_1,x_2\}$, $(1-h)V=kx_3$.  Obviously, $(1-g)V\cap(1-h)V=0$.  We also have respective generators
$$
\omega_g=\partial_1,\ \ \omega_h=\partial_3,\ \ \text{and}\ \ \omega_{gh}=\partial_1\partial_3
$$
of the highest wedge powers of $\big((1-g)V\big)^\ast$, $\big((1-h)V\big)^\ast$, and $\big((1-gh)V\big)^\ast$ respectively.  Thus $\det^\perp_g=k\omega_gg$, $\det^\perp_h=k\omega_h h$, and $\det^\perp_{gh}=k\omega_{gh} gh$.
\par

Consider the degree $2$ cochains
$$
X=\omega_g\partial_2g\in S(V^g)\ox (V^g)^\ast{\det}^\perp_g
$$
and
$$
Y=x_2\omega_h\partial_2h\in S(V^h)\ox (V^h)^\ast{\det}^\perp_h.
$$
Then one can easily see, directly from Lemma \ref{lem:circles}, that applying the bilinear operation $[\ ,\ ]_\phi$ produces
$$
[X,Y]_\phi=\omega_g\omega_h\partial_2gh=\omega_{gh}\partial_2gh\in S(V^{gh})\ox(V^{gh})^\ast{\det}^\perp_{gh}.
$$
This would appear to contradict the degree 2 vanishing result of \cite{SW2}, but it actually does not!  The point is that neither $X$ nor $Y$ is invariant.  In fact, $X\cdot\int_G=Y\cdot\int_G=0$.
\end{example}

\begin{example}\label{ex:two}
Let $G=\mathbb{Z}/N\mathbb{Z}\times\mathbb{Z}/M\mathbb{Z}$ for integers $N,M>1$.  We assume $k=\bar{k}$, or that $M=N=2$.  Let $\sigma$ and $\tau$ be the generators of $\mathbb{Z}/N\mathbb{Z}$ and $\mathbb{Z}/M\mathbb{Z}$ respectively.  Take $W=k\{x_1,x_2,x_3,x_4,x_5\}$ and embed $G$ in $GL(W)$ by identifying $\sigma$ and $\tau$ with the diagonal matrices
$$
\sigma=\mathrm{diag}\{\zeta,\zeta^{-1},1,1,1\},\ \ 
\tau=\mathrm{diag}\{1,1,1,\vartheta^{-1},\vartheta\},
$$
where $\zeta$ and $\vartheta$ are primitive $N$th and $M$th roots of $1$ in $k$.
\par

We have
$$
(1-\sigma)=\mathrm{diag}\{(1-\zeta),(1-\zeta^{-1}),0,0,0\},\ \ (1-\tau)=\mathrm{diag}\{0,0,0,(1-\vartheta^{-1}),(1-\vartheta)\},
$$
and these are both rank $2$ matrices.  More specifically, $(1-\sigma)W=k\{x_1,x_2\}$ and $(1-\tau)W=k\{x_4,x_5\}$.  So $(W^\sigma)^\perp\cap (W^\tau)^\perp=0$.  Similarly we have
$$
\sigma\tau=\mathrm{diag}\{\zeta,\zeta^{-1},1,\vartheta^{-1},\vartheta\},\ \ (1-\sigma\tau)=\mathrm{diag}\{(1-\zeta),(1-\zeta^{-1}),0,(1-\vartheta^{-1}),(1-\vartheta)\},
$$
$$
(1-\sigma\tau)W=k\{x_1,x_2,x_4,x_5\}.
$$
\par

We take
$$
\omega_\sigma=\partial_1\partial_2,\ \ \omega_\tau=\partial_4\partial_5,\ \ \omega_{\sigma\tau}=\partial_1\partial_2\partial_4\partial_5,
$$
and $X=\omega_{\sigma}\partial_3\sigma$, $Y=x_3\omega_\tau\tau$.  In this case $X$ and $Y$ are $G$-invariant,  and hence represent classes in $\mathrm{HH}^\bullet(S(V)\# G)$.  One then produces via Theorem \ref{thm:gh} the nonvanishing Gerstenhaber bracket
$$
[X,Y]=\omega_{\sigma\tau}\sigma\tau\in \mathrm{HH}^4(S(V)\# G).
$$
This example can be generalized  easily to produce nonzero brackets in higher degree.
\end{example}

\subsection{Vanishing of brackets in degree 2}

Consider the subcomplex 
\[
   \mathscr{D}(1)\subset\bigoplus_g S(V^g)\ox \Wedge^{\bu-\codim V^g}(V^g)^\ast\det_g^\perp
\]
consisting of all summands corresponding to group elements $g$ with $\codim V^g=1$.  This subcomplex is stable under the $G$-action.  It was seen already in \cite{farinati,SW2} that $\mathscr{D}(1)^G=0$.  So actually, after we take invariance, we have
$$
\HH^\bu(S(V)\# G)=\big(\bigoplus_g S(V^g)\ox \Wedge^{\bu-\codim V^g}(V^g)^\ast\det_g^\perp\big)^G=\mathscr{D}(0)^G\oplus \mathscr{D}(>1)^G.
$$
where in $\mathscr{D}(>1)$ we have all summands corresponding to $g$ with $\mathrm{codim}V^g>1$.
\par

It follows that, after we take invariants and restrict ourselves to considering only elements in degree $2$, the only situations that can occur when taking brackets in $\mathscr{D}(>0)^G=\mathscr{D}(>1)^G$ are covered by Corollary \ref{thm:supp-off-id}.  Hence when we apply the Gerstenhaber bracket we get
$$
\left[\left(\mathscr{D}(>0)^2\right)^G,\left(\mathscr{D}(>0)^2\right)^G\right]=0.
$$
This rephrases the argument given in \cite[Theorem 9.2]{SW2}.


\appendix
\section{}

This appendix is dedicated to giving details of the proof of Lemma~\ref{lemma:phimap}. 
Take $\underline{n}=\{1,\dots, n\}$.
We first state the lemma in a slightly different way. 

\begin{proposition}\label{prop:phimap}
There are coefficients $\xi_r^{s,t,z}\in\mathbb{Q}$ such that the degree $-1$ $A^e$-linear map $\phi:K\ox_A K\to K$ given by
$$
\phi\big(1\ox o(u_1,\dots, u_s)\ox v_1\dots v_t\ox o(w_1,\dots, w_z)\ox 1\big)
$$
$$
=\sum_{\sigma\in S_t,\ r\in \underline{t}} (-1)^{sz+z}\xi_r^{s,t,z}v_{\sigma(1)}\dots v_{\sigma(r-1)}\ox o(w_1,\dots, w_z,v_{\sigma(r)},u_1,\dots, u_s)\ox v_{\sigma(r+1)}\dots v_{\sigma(t)}.
$$
satisfies
$$
d_K\phi+\phi d_{K\ox_A K}=F_K.
$$
If we take all the coefficients for $\phi|K_0\ox_A K_0$ to be $\xi_r^{0,t,0}=\frac{1}{t!}$, and suppose $\dim V$ is arbitrarily large, then the $\xi_r^{s,t,z}$ are specified uniquely as
\begin{equation}\label{0}
\xi^{s,t,z}_r=\frac{(r+z-1)!(t-r+s)!}{(r-1)!(t-r)!(s+t+z)!},
\end{equation}
\end{proposition}

We will show that the coefficients $\xi_r^{s,t,z}$ given at (\ref{0}) produce a map $\phi$ with the necessary property.  We note that a sum indexed over the empty set is $0$.
\par

We check the formula on the subcomplexes $K_{\geq 0}\ox_A K_0$ and $K_0\ox_A K_{\geq 0}$.  On these subcomplexes the coefficients $\xi^{s,t,0}_r$ and $\xi^{0,s,z}_r$ are given by
$$
\xi^{s,t,0}_r=\frac{(t-r+s)!}{(t-r)!(s+t)!}\ \ \mathrm{and}\ \ \xi^{0,t,z}_r=\frac{(r+z-1)!}{(r-1)!(t+z)!}
$$
respectively.  We then proceed to check the formula of $K_{>0}\ot K_{>0}$.  One can check easily that in degree $0$ we have
$$
d\phi_0=\mu\ox id-id\ox \mu:A\ox A\ox A.
$$
So we will only be considering elements of (total) degree $\geq 1$ below.
\par

Take a monomial
$$
1\ox o(u_1\dots u_s)\ox v_1\dots v_t\ox 1
$$
in $K_{> 0}\ox_A K_0$.  Let $a=v_1\dots v_t$ Then we have
$$
\ba{l}
d\phi(1\ox o(u_1\dots u_s)\ox v_1\dots v_t\ox 1)\\
=d(\sum \xi_r^{s,t,0}v_{\sigma(1)}\dots v_{\sigma(r-1)}\ox o(v_{\sigma(r)},u_1\dots u_s)\ox v_{\sigma(r+1)}\dots v_{\sigma(t)})\\
\\
=\sum_{r,\sigma,l}(-1)^{l}\xi_r^{s,t,0}v_{\sigma(1)}\dots v_{\sigma(r-1)}u_l\ox o(v_{\sigma(r)},u_1\dots\hat{u_l}\dots u_s)\ox v_{\sigma(r+1)}\dots v_{\sigma(t)}\\
-\sum_{r,\sigma,l}(-1)^{l}\xi_r^{s,t,0}v_{\sigma(1)}\dots v_{\sigma(r-1)}\ox o(v_{\sigma(r)},u_1\dots\hat{u_l}\dots u_s)\ox u_lv_{\sigma(r+1)}\dots v_{\sigma(t)}\\
+\sum_{r,\sigma}\xi_r^{s,t,0}v_{\sigma(1)}\dots v_{\sigma(r-1)}v_{\sigma(r)}\ox o(u_1\dots u_s)\ox v_{\sigma(r+1)}\dots v_{\sigma(t)}\\
-\sum_{r,\sigma}\xi_r^{s,t,0}v_{\sigma(1)}\dots v_{\sigma(r-1)}\ox o(u_1\dots u_s)\ox v_{\sigma(r)}v_{\sigma(r+1)}\dots v_{\sigma(t)}
\ea
$$
\begin{equation}\label{1}
\ba{l}
=t!\xi_{t}^{s,t,0}a\ox o(u_1\dots u_s)\ox 1-t!\xi_{1}^{s,t,0}1\ox o(u_1\dots u_s)\ox a\\
+\sum_{r,\sigma,l}(-1)^{l}\xi_r^{s,t,0}v_{\sigma(1)}\dots v_{\sigma(r-1)}u_l\ox o(v_{\sigma(r)},u_1\dots\hat{u_l}\dots u_s)\ox v_{\sigma(r+1)}\dots v_{\sigma(t)}\\
-\sum_{r,\sigma,l}(-1)^{l}\xi_r^{s,t,0}v_{\sigma(1)}\dots v_{\sigma(r-1)}\ox o(v_{\sigma(r)},u_1\dots\hat{u_l}\dots u_s)\ox u_lv_{\sigma(r+1)}\dots v_{\sigma(t)}\\
+\sum_{r\neq t,\sigma}(\xi_r^{s,t,0}-\xi_{r+1}^{s,t,0})v_{\sigma(1)}\dots v_{\sigma(r)}\ox o(u_1\dots u_s)\ox v_{\sigma(r+1)}\dots v_{\sigma(t)}.
\ea
\end{equation}
\par

On the other hand,
$$
\ba{l}
\phi d(1\ox o(u_1\dots u_s)\ox v_1\dots v_t\ox 1)\\
=\phi(\sum_l (-1)^{l-1}u_l\ox o(u_1\dots \hat{u_l}\dots u_s)\ox v_1\dots v_t\ox 1)\\
-\phi(\sum_l(-1)^{l-1}\ox o(u_1\dots \hat{u_l}\dots u_s)\ox u_lv_1\dots v_t\ox 1)\\
\\
=\phi(\sum_l (-1)^{l-1}u_l\ox o(u_1\dots \hat{u_l}\dots u_s)\ox v_1\dots v_t\ox 1)\\
-\phi(\sum_l(-1)^{l-1}\ox o(u_1\dots \hat{u_l}\dots u_s)\ox v_1\dots v_tv_{t+1}\ox 1)\ \ \ \mathrm{where\ }v_{t+1}=u_l\\
\\
=\sum_{r,\sigma,l}(-1)^{l-1}\xi_r^{s-1,t,0}v_{\sigma(1)}\dots v_{\sigma(r-1)}u_l\ox o(v_{\sigma(r)},u_1\dots\hat{u_l}\dots u_s)\ox v_{\sigma(r+1)}\dots v_{\sigma(t)}\\
-\sum_{\mu\leq t+1,\xi\in S_{t+1},l}(-1)^{l-1}\xi_r^{s-1,t+1,0}v_{\xi(1)}\dots v_{\xi(\mu-1)}\ox o(v_{\xi(\mu)},u_1\dots\hat{u_l}\dots u_s)\ox v_{\xi(\mu+1)}\dots v_{\sigma(t)}v_{\xi(t+1)}\\
\\
=\sum_{r,\sigma,l}(-1)^{l-1}\xi_r^{s-1,t,0}v_{\sigma(1)}\dots v_{\sigma(r-1)}u_l\ox o(v_{\sigma(r)},u_1\dots\hat{u_l}\dots u_s)\ox v_{\sigma(r+1)}\dots v_{\sigma(t)}\\
-\sum_{1\leq r\leq t,\sigma\in S_{t},l} (-1)^{l-1}r\xi_{r+1}^{s-1,t+1,0}v_{\sigma(1)}\dots v_{\sigma(r-1)}u_l\ox o(v_{\sigma(r)},u_1\dots\hat{u_l}\dots u_s)\ox v_{\sigma(r+1)}\dots v_{\sigma(t)}\\
-\sum_{1\leq r\leq t,\sigma\in S_{t},l}(-1)^{l-1}(t-r+1)\xi_{r}^{s-1,t+1,0}v_{\sigma(1)}\dots v_{\sigma(r-1)}\ox o(v_{\sigma(r)},u_1\dots\hat{u_l}\dots u_s)\ox u_lv_{\sigma(r+1)}\dots v_{\sigma(t)}\\
-\sum_{0\leq r\leq t,\sigma\in S_{t},l}s\xi_{r+1}^{s-1,t+1,0}v_{\sigma(1)}\dots v_{\sigma(r)}\ox o(u_1\dots u_s)\ox v_{\sigma(r+1)}\dots v_{\sigma(t)}
\ea
$$

\begin{equation}\label{2}
\ba{l}
=-t!s\xi_{t+1}^{s-1,t+1,0}a\ox o(u_1\dots u_s)\ox 1-t!s\xi_1^{s-1,t+1,0}1\ox o(u_1\dots u_s)\ox a\\
+\sum_{r,\sigma,l}(-1)^{l-1}(\xi_r^{s-1,t,0}-r\xi_{r+1}^{s-1,t+1,0})v_{\sigma(1)}\dots v_{\sigma(r-1)}u_l\ox o(v_{\sigma(r)},u_1\dots\hat{u_l}\dots u_s)\ox v_{\sigma(r+1)}\dots v_{\sigma(t)}\\
-\sum_{r,\sigma,l}(-1)^{l-1}(t-r+1)\xi_{r}^{s-1,t+1,0}v_{\sigma(1)}\dots v_{\sigma(r-1)}\ox o(v_{\sigma(r)},u_1\dots\hat{u_l}\dots u_s)\ox u_lv_{\sigma(r+1)}\dots v_{\sigma(t)}\\
-\sum_{r\neq t,\sigma}s\xi_{r+1}^{s-1,t+1,0}v_{\sigma(1)}\dots v_{\sigma(r)}\ox o(u_1\dots u_s)\ox v_{\sigma(r+1)}\dots v_{\sigma(t)},\\
\ea
\end{equation}
where the final expressions are indexed over $r\in \underline{t},\sigma\in S_t,l\in\underline{s}$.
\par

By matching the coefficients for (\ref{1}) and (\ref{2}) we see that the equation $d\phi+\phi d=-id_K\ox \mu$ demands a number of equations:
\begin{enumerate}
\item[lEQ 1:] $\xi_{t}^{s,t,0}-s\xi_{t+1}^{s-1,t+1,0}=0$.
\item[lEQ 2:] $\xi_{1}^{s,t,0}+s\xi_1^{s-1,t+1,0}=\frac{1}{t!}$.
\item[lEQ 3:] $\xi_r^{s,t,0}=(\xi_r^{s-1,t,0}-r\xi_{r+1}^{s-1,t+1,0})$.
\item[lEQ 4:] $\xi_r^{s,t,0}=(t-r+1)\xi_{r}^{s-1,t+1,0}$.
\item[lEQ 5:] $(\xi_r^{s,t,0}-\xi_{r+1}^{s,t,0})=s\xi_{r+1}^{s-1,t+1,0}$ when $r\neq t$.
\end{enumerate}
Note that when the $u_i$ and $v_i$ are linearly independent, the $\xi_r^{s,t,0}$ will be specified uniquely by lEQ 3.  If the above $5$ equations are satisfied then we will have $(\ref{1})+(\ref{2})=-id\ox \mu$.

\begin{lemma}
The coefficients $\xi_r^{s,t,0}$ satisfy each of the equations \rm{lEQ} $i$.
\end{lemma}

\begin{proof}
We simply check each equation.
\par
$$
\frac{(t-r+s)!}{(t-r)!(s+t)!}
$$
For lEQ $1$:
$$
\ba{rll}
\xi_{t}^{s,t,0}-s\xi_{t+1}^{s-1,t+1,0}&=\frac{s!}{(s+t)!}-s\frac{(s-1)!}{(s+t)!}\\
&=\frac{s!}{(s+t)!}-\frac{s!}{(s+t)!}\\
&=0. & \surd
\ea
$$
For lEQ $2$:
$$
\ba{rll}
\xi_{1}^{s,t,0}+s\xi_1^{s-1,t+1,0}&=\frac{(t-1+s)!}{(t-1)!(s+t)!}+s\frac{(t+s-1)!}{t!(s+t)!}\\
&=\frac{(t+s)(t+s-1)!}{t!(s+t)!}=\frac{(t+s)!}{t!(s+t)!}= \frac{1}{t!}. &\surd
\ea
$$
For lEQ $3$:
$$
\ba{rll}
\xi_r^{s-1,t,0}-r\xi_{r+1}^{s-1,t+1,0}&=\frac{(t-r+s-1)!}{(t-r)!(s+t-1)!}-r\frac{(t-r+s-1)!}{(t-r)!(s+t)!}\\
&=\frac{(s+t-r)(t-r+s-1)!}{(t-r)!(s+t)!}=\frac{(s+t-r)!}{(t-r)!(s+t)!}=\xi_r^{s,t,0}. & \surd
\ea
$$
For lEQ $4$:
$$
(t-r+1)\xi_{r}^{s-1,t+1,0}=(t-r+1)\frac{(s+t-r)!}{(t-r+1)!(s+t)!}=\frac{(s+t-r)!}{(t-r)!(s+t)!}=\xi_r^{s,t,0}.\ \ \surd
$$
For lEQ $5$:
$$
\ba{rll}
\xi_r^{s,t,0}-\xi_{r+1}^{s,t,0} &=\frac{(s+t-r)!}{(t-r)!(s+t)!}-\frac{(s+t-r-1)!}{(t-r-1)!(s+t)!}\\
&=\frac{(s+t-r)!-(t-r)(s+t-r-1)!}{(t-r)!(s+t)!}\\
&=\frac{(s+t-r)(s+t-r-1)!-(t-r)(s+t-r-1)!}{(t-r)!(s+t)!}\\
&=\frac{s(s+t-r)!}{(t-r)!(s+t)!}\\
&=s\xi_{r+1}^{s-1,t+1,0}. & \surd
\ea
$$
\end{proof}

\begin{corollary}
The equation
$$
d\phi-\phi d=\mu\ox id-id\ox \mu=F_K
$$
on all of $K\ox_A K_0$.
\end{corollary}

\begin{proof}
We already know that the equation is satisfied on $K_0\ox_A K_0$, and the satisfaction of the above equations lEQ 1-5 ensures that
$$
-id\ox \mu=(\ref{1})+(\ref{2})=(d\phi+\phi d)|K_{>0}\ox_A K_0.
$$
Since $\mu\ox id|K_{>0}\ox_A K_0$, this gives the equality
$$
d\phi-\phi d=\mu\ox id-id\ox \mu
$$
on all of $K\ox_A K_0$.
\end{proof}

Let $z\geq 1$.  We now address the proposed equality $d\phi+\phi d=F_K$ on $K_0\ox K_{>0}$.  We have
$$
\ba{l}
d\phi(1\ox v_1\dots v_t\ox o(w_1\dots w_z)\ox 1)\\
d(\sum_{r,\sigma}(-1)^{z}\xi_r^{0,t,z}v_{\sigma(1)}\dots v_{\sigma(r-1)}\ox o(w_1\dots w_z,v_{\sigma(r)})\ox v_{\sigma(r+1)}\dots v_{\sigma(t)})\\
\\
= \sum_{r,\sigma,l}(-1)^{z+l-1}\xi_r^{0,t,z}v_{\sigma(1)}\dots v_{\sigma(r-1)}w_l\ox o(w_1\dots\hat{w_l}\dots w_z,v_{\sigma(r)})\ox v_{\sigma(r+1)}\dots v_{\sigma(t)}\\
-\sum_{r,\sigma,l}(-1)^{z+l-1}\xi_r^{0,t,z}v_{\sigma(1)}\dots v_{\sigma(r-1)}\ox o(w_1\dots\hat{w_l}\dots w_z,v_{\sigma(r)})\ox w_lv_{\sigma(r+1)}\dots v_{\sigma(t)}\\
+\sum_{r,\sigma,l}\xi_r^{0,t,z}v_{\sigma(1)}\dots v_{\sigma(r-1)}v_{\sigma(r)}\ox o(w_1\dots w_z)\ox v_{\sigma(r+1)}\dots v_{\sigma(t)}\\
-\sum_{r,\sigma,l}\xi_r^{0,t,z}v_{\sigma(1)}\dots v_{\sigma(r-1)}\ox o(w_1\dots w_z)\ox v_{\sigma(r)}v_{\sigma(r+1)}\dots v_{\sigma(t)}
\ea
$$

\begin{equation}\label{3}
\ba{l}
=t!\xi_t^{0,t,z}a\ox o(w_1\dots w_z)\ox 1-t!\xi_1^{0,t,z}1\ox o(w_1,\dots, w_z)\ox a\\
+\sum_{r,\sigma,l}(-1)^{z+l-1}\xi_r^{0,t,z}v_{\sigma(1)}\dots v_{\sigma(r-1)}w_l\ox o(w_1\dots\hat{w_l}\dots w_z,v_{\sigma(r)})\ox v_{\sigma(r+1)}\dots v_{\sigma(t)}\\
-\sum_{r,\sigma,l}(-1)^{z+l-1}\xi_r^{0,t,z}v_{\sigma(1)}\dots v_{\sigma(r-1)}\ox o(w_1\dots\hat{w_l}\dots w_z,v_{\sigma(r)})\ox w_lv_{\sigma(r+1)}\dots v_{\sigma(t)}\\
+\sum_{r\neq 1,\sigma,l}(\xi_{r-1}^{0,t,z}-\xi_r^{0,t,z})v_{\sigma(1)}\dots v_{\sigma(r-1)}\ox o(w_1\dots w_z)\ox v_{\sigma(r)}\dots v_{\sigma(t)}.
\ea
\end{equation}
\par

On the other hand we have
$$
\ba{l}
\phi d(1\ox v_1\dots v_t\ox o(w_1\dots w_z)\ox 1)\\
=\phi(\sum_l(-1)^{l-1}v_1\dots v_tv_{t+1}\ox o(w_1\dots \hat{w_l}\dots w_z)\ox 1)\ \ \mathrm{where\ }v_{t+1}=w_l\\
-\phi(\sum_l(-1)^{l-1}v_1\dots v_t\ox o(w_1\dots \hat{w_l}\dots w_z)\ox w_l)\\
\\
=\sum_{r,\sigma, l}(-1)^{z+l}r\xi_{r+1}^{0,t+1,z-1}v_{\sigma(1)}\dots v_{\sigma(r-1)}w_l\ox o(w_1\dots\hat{w_l}\dots w_z,v_{\sigma(r)})\ox v_{\sigma(r+1)}\dots v_{\sigma(t)}\\
+\sum_{r,\sigma, l}(-1)^{z+l}(t-r+1)\xi_{r}^{0,t+1,z-1}v_{\sigma(1)}\dots v_{\sigma(r-1)}\ox o(w_1\dots\hat{w_l}\dots w_z,v_{\sigma(r)})\ox w_lv_{\sigma(r+1)}\dots v_{\sigma(t)}\\
+\sum_{r\leq t+1,\sigma}z\xi_r^{0,t+1,z-1}v_{\sigma(1)}\dots v_{\sigma(r-1)}\ox o(w_1\dots w_z)\ox v_{\sigma(r)}\dots v_{\sigma(t)}\\
-\sum_{r,\sigma,l}(-1)^{z+l}\xi_r^{0,t,z-1}v_{\sigma(1)}\dots v_{\sigma(r-1)}\ox o(w_1\dots \hat{w_l}\dots w_z,v_{\sigma(r)})\ox w_lv_{\sigma(r+1)}\dots v_{\sigma(t)}\\
\ea
$$
\begin{equation}\label{4}
\ba{l}
=t!z\xi_{t+1}^{0,t+1,z-1}a\ox o(w_1\dots w_z)\ox 1+t!z\xi_{1}^{0,t+1,z-1}1\ox o(w_1\dots w_z)\ox a\\
+\sum_{r,\sigma, l}(-1)^{z+l}r\xi_{r+1}^{0,t+1,z-1}v_{\sigma(1)}\dots v_{\sigma(r-1)}w_l\ox o(w_1\dots\hat{w_l}\dots w_z,v_{\sigma(r)})\ox v_{\sigma(r+1)}\dots v_{\sigma(t)}\\
-\sum_{r,\sigma, l}(-1)^{z+l}(\xi_r^{0,t,z-1}-(t-r+1)\xi_{r}^{0,t+1,z-1}) v_{\sigma(1)}\dots v_{\sigma(r-1)}\ox o(w_1\dots\hat{w_l}\dots w_z,v_{\sigma(r)})\ox w_lv_{\sigma(r+1)}\dots v_{\sigma(t)}\\
+\sum_{r\neq 1,\sigma}z\xi_r^{0,t+1,z-1}v_{\sigma(1)}\dots v_{\sigma(r-1)}\ox o(w_1\dots w_z)\ox v_{\sigma(r)}\dots v_{\sigma(t)}
\ea
\end{equation}

We will have $(\ref{3})+(\ref{4})=\mu\ox id$ if the $\xi_r^{0,t,z}$ satisfy the following equations:
\begin{enumerate}
\item[rEQ 1:] $\xi_t^{0,t,z}+z\xi_{t+1}^{0,t+1,z-1}=\frac{1}{t}$.
\item[rEQ 2:] $\xi_1^{0,t,z}-z\xi_1^{0,t+1,z-1}=0$.
\item[rEQ 3:] $r\xi_{r+1}^{0,t+1,z-1}=\xi_r^{0,t,z}$.
\item[rEQ 4:] $\xi_r^{0,t,z-1}-(t-r+1)\xi_r^{0,t+1,z-1}=\xi_r^{0,t,z}$.
\item[rEQ 5:] $\xi_{r-1}^{0,t,z}-\xi_r^{0,t,z}=-z\xi_r^{0,t+1,z-1}$.
\end{enumerate}
Indeed, if $\dim V$ is sufficiently large, rEQ 3 will specify the $\xi_r^{0,t,z}$ from the lower degree terms $\xi_{r+1}^{0,t+1,z-1}$.

\begin{lemma}
The coefficients $\xi_r^{0,t,z}$ satisfy all of the above equations {\rm rEQ} i.
\end{lemma}

\begin{proof}
For rEQ 1:
$$
\xi_t^{0,t,z}+z\xi_{t+1}^{0,t+1,z-1}=\frac{(t+z-1)!}{(t-1)!(t+z)!}+z\frac{(t+z-1)!}{t!(t+z)!}=\frac{(t+z)(t+z-1)!}{t!(t+z)!}=\frac{1}{t!}.\ \ \ \surd
$$
For rEQ 2:
$$
\xi_1^{0,t,z}-z\xi_1^{0,t+1,z-1}=\frac{z!}{(t+z)!}-z\frac{(z-1)!}{(t+z)!}=0.\ \ \ \surd
$$
For rEQ 3:
$$
r\xi_{r+1}^{0,t+1,z-1}=r\frac{(r+z-1)!}{r!(t+z)!}=\frac{(r+z-1)!}{(r-1)!(t+z)!}=\xi_r^{0,t,z}.\ \ \ \surd
$$
For rEQ 4:
$$
\ba{rl}
\xi_r^{0,t,z-1}-(t-r+1)\xi_r^{0,t+1,z-1}=\frac{(r+z-2)!}{(r-1)!(t+z-1)!}-(t-r+1)\frac{(r+z-2)!}{(r-1)!(t+z)!}=\frac{(z+r-1)(r+z-2)!}{(r-1)!(t+z)!}=\xi_r^{0,t,z}.\ \ \ \surd
\ea
$$
For rEQ 5:
$$
\ba{rl}
\xi_{r-1}^{0,t,z}-\xi_r^{0,t,z}=\frac{(r+z-2)!}{(r-2)!(t+z)!}-\frac{(r+z-1)!}{(r-1)!(t+1)!}=\frac{(r-1-r-z-1)(r+z-2)!}{(r-1)!(t+z)!}=-z\frac{(r+z-2)!}{(r-1)!(t+z)!}=-z\xi_r^{0,t+1,z-1}.\ \ \ \surd
\ea
$$
\end{proof}

We get again the desired property.

\begin{corollary}
The equation
$$
d\phi-\phi d=\mu\ox id-id\ox \mu=F_K
$$
on all of $K_0\ox_A K$.
\end{corollary}

\begin{proof}
The equations rEQ 1-5 imply the equality
$$
(d\phi+\phi d)|K_0\ox_A K_{>0}=(\ref{3})+(\ref{4})=\mu\ox 1=F_K|K_0\ox_A K_{>0}.
$$
We have already checked the equality in degree $0$.
\end{proof}

Now we assume $s>0$ and $z>0$ and consider a monomial
\[
1\ox o(u_1,\dots, u_s)\ox v_1\dots v_t\ox o(w_1,\dots, w_z)\ox 1.
\]
For $d\phi$ we have
\[
\ba{l}
d\phi(1\ox o(u_1,\dots, u_s)\ox v_1\dots v_t\ox o(w_1,\dots, w_z)\ox 1)\\
=\sum_{\sigma\in S_t,\ r\in \underline{t}} (-1)^{sz+z}\xi_r^{s,t,z}d(v_{\sigma(1)}\dots v_{\sigma(r-1)}\ox o(w_1,\dots, w_z,v_{\sigma(r)},u_1,\dots, u_s)\ox v_{\sigma(r+1)}\dots v_{\sigma(t)})\\
=-\sum_{\sigma,r\in \underline{t},l\in \unl{z}}(-1)^{sz+z+l}\xi_r^{s,t,z}v_{\sigma(1)}\dots v_{\sigma(r-1)}w_l\ox o(w_1,\dots,\hat{w_l},\dots,w_z,v_{\sigma(r)},u_1,\dots, u_s)\ox v_{\sigma(r+1)}\dots v_{\sigma(t)}\\
+\sum_{\sigma,r\in \underline{t},l\in \unl{z}} (-1)^{sz+z+l}\xi_r^{s,t,z}v_{\sigma(1)}\dots v_{\sigma(r-1)}\ox o(w_1,\dots,\hat{w_l},\dots,w_z,v_{\sigma(r)},u_1,\dots, u_s)\ox w_lv_{\sigma(r+1)}\dots v_{\sigma(t)})\\
+\sum_{\sigma,r\in \underline{t},l\in \unl{s}}(-1)^{sz+z+l+z}\xi_r^{s,t,z}v_{\sigma(1)}\dots v_{\sigma(r-1)}u_l\ox o(w_1,\dots,w_z,v_{\sigma(r)},u_1,\dots,\hat{u_l},\dots,u_s)\ox v_{\sigma(r+1)}\dots v_{\sigma(t)}\\
-\sum_{\sigma,r\in \underline{t},l\in \unl{s}}(-1)^{sz+z+l+z}\xi_r^{s,t,z}v_{\sigma(1)}\dots v_{\sigma(r-1)}\ox o(w_1,\dots,w_z,v_{\sigma(r)},u_1,\dots,\hat{u_l},\dots,u_s)\ox u_lv_{\sigma(r+1)}\dots v_{\sigma(t)}\\
+\sum_{\sigma,r\in \underline{t}}(-1)^{sz}\xi_r^{s,t,z}v_{\sigma(1)}\dots v_{\sigma(r)}\ox o(w_1,\dots,w_z,u_1,\dots,u_s)\ox v_{\sigma(r+1)}\dots v_{\sigma(t)}\\
-\sum_{\sigma,r\in \underline{t}}(-1)^{sz}\xi_r^{s,t,z}v_{\sigma(1)}\dots v_{\sigma(r-1)}\ox o(w_1,\dots,w_z,u_1,\dots,u_s)\ox v_{\sigma(r)}\dots v_{\sigma(t)}
\ea
\]
\begin{equation}\label{eq:finals1}
\ba{l}
=-\sum_{\sigma,r\in \underline{t},l\in \unl{z}}(-1)^{sz+z+l}\xi_r^{s,t,z}v_{\sigma(1)}\dots v_{\sigma(r-1)}w_l\ox o(w_1,\dots,\hat{w_l},\dots,w_z,v_{\sigma(r)},u_1,\dots, u_s)\ox v_{\sigma(r+1)}\dots v_{\sigma(t)}\\
+\sum_{\sigma,r\in \underline{t},l\in \unl{z}} (-1)^{sz+z+l}\xi_r^{s,t,z}v_{\sigma(1)}\dots v_{\sigma(r-1)}\ox o(w_1,\dots,\hat{w_l},\dots,w_z,v_{\sigma(r)},u_1,\dots, u_s)\ox w_lv_{\sigma(r+1)}\dots v_{\sigma(t)})\\
+\sum_{\sigma,r\in \underline{t},l\in \unl{s}}(-1)^{sz+l}\xi_r^{s,t,z}v_{\sigma(1)}\dots v_{\sigma(r-1)}u_l\ox o(w_1,\dots,w_z,v_{\sigma(r)},u_1,\dots,\hat{u_l},\dots,u_s)\ox v_{\sigma(r+1)}\dots v_{\sigma(t)}\\
-\sum_{\sigma,r\in \underline{t},l\in \unl{s}}(-1)^{sz+l}\xi_r^{s,t,z}v_{\sigma(1)}\dots v_{\sigma(r-1)}\ox o(w_1,\dots,w_z,v_{\sigma(r)},u_1,\dots,\hat{u_l},\dots,u_s)\ox u_lv_{\sigma(r+1)}\dots v_{\sigma(t)}\\
+(-1)^{sz}t!\xi_t^{s,t,z}v_{1}\dots v_{t}\ox o(w_1,\dots,w_z,u_1,\dots,u_s)\ox 1\\
-(-1)^{sz}t!\xi_1^{s,t,z}1\ox o(w_1,\dots,w_z,u_1,\dots,u_s)\ox v_{1}\dots v_{t}\\
+\sum_{\sigma,0<r<t}(-1)^{sz}(\xi_{r}^{s,t,z}-\xi_{r+1}^{s,t,z})v_{\sigma(1)}\dots v_{\sigma(r)}\ox o(w_1,\dots,w_z,u_1,\dots,u_s)\ox v_{\sigma(r+1)}\dots v_{\sigma(t)}.
\ea
\end{equation}
On the other hand
\[
\ba{l}
\phi d(1\ox o(u_1,\dots, u_s)\ox v_1\dots v_t\ox o(w_1,\dots, w_z)\ox 1)\\
=-\sum_{l\in\unl{s}}(-1)^{l}u_l\phi(1\ox o(u_1,\dots,\hat{u_l},\dots, u_s)\ox v_1\dots v_t\ox o(w_1,\dots, w_z)\ox 1)\\
+\sum_{l\in\unl{s}}(-1)^{l}\phi(1\ox o(u_1,\dots,\hat{u_l},\dots, u_s)\ox u_lv_1\dots v_t\ox o(w_1,\dots, w_z)\ox 1)\\
-\sum_{l\in\unl{z}}(-1)^{l+s}\phi(1\ox o(u_1,\dots, u_s)\ox v_1\dots v_tw_l\ox o(w_1,\dots,\hat{w_l}\dots, w_z)\ox 1)\\
+\sum_{l\in\unl{z}}(-1)^{l+s}\phi(1\ox o(u_1,\dots, u_s)\ox v_1\dots v_t\ox o(w_1,\dots,\hat{w_l}\dots, w_z)\ox 1)w_l\\
\ea
\]
\[
\ba{l}
=-\sum_{\sigma,l\in\unl{s}}(-1)^{(s-1)z+z+l}\xi_r^{s-1,t,z}v_{\sigma(1)}\dots v_{\sigma(r-1)}u_l\ox o(w_1,\dots,w_z,v_{\sigma(r)},u_1,\dots,\hat{u_l},\dots,u_s)\ox v_{\sigma(r+1)}\dots v_{\sigma(t)}\\
+\sum_{\sigma,l\in\unl{s}} (-1)^{(s-1)z+z+l}\xi_{r+1}^{s-1,t+1,z}v_{\sigma(1)}\dots v_{\sigma(r)}\ox o(w_1,\dots,w_z,u_l,u_1,\dots,\hat{u_l},\dots,u_s)\ox v_{\sigma(r+1)}\dots v_{\sigma(t)}\\
+\sum_{\sigma,l\in\unl{s}} (-1)^{(s-1)z+z+l}r\xi_{r+1}^{s-1,t+1,z}v_{\sigma(1)}\dots v_{\sigma(r-1)}u_l\ox o(w_1,\dots,w_z,v_{\sigma(r)},u_1,\dots,\hat{u_l},\dots,u_s)\ox v_{\sigma(r+1)}\dots v_{\sigma(t)}\\
+\sum_{\sigma,l\in\unl{s}} (-1)^{(s-1)z+z+l}(t-r+1)\xi_{r}^{s-1,t+1,z}v_{\sigma(1)}\dots v_{\sigma(r-1)}\ox o(w_1,\dots,w_z,v_{\sigma(r)},u_1,\dots,\hat{u_l},\dots,u_s)\ox u_l v_{\sigma(r+1)}\dots v_{\sigma(t)}\\
+\sum_{l\in\unl{z}}(-1)^{s(z-1)+(z-1)+l+s}\xi^{s,t,z-1}_r v_{\sigma(1)}\dots v_{\sigma(r-1)}\ox o(w_1,\dots,\hat{w_l},\dots,w_z,v_{\sigma(r)},u_1,\dots, u_s)\ox w_lv_{\sigma(r+1)}\dots v_{\sigma(t)}\\
-\sum_{\sigma,l\in\unl{s}} (-1)^{s(z-1)+(z-1)+l+s}r\xi_{r+1}^{s,t+1,z-1}v_{\sigma(1)}\dots v_{\sigma(r-1)}w_l\ox o(w_1,\dots,\hat{w_l},\dots ,w_z,v_{\sigma(r)},u_1,\dots,u_s)\ox v_{\sigma(r+1)}\dots v_{\sigma(t)}\\
-\sum_{\sigma,l\in\unl{s}} (-1)^{s(z-1)+(z-1)+l+s}(t-r+1)\xi_{r}^{s,t+1,z-1}v_{\sigma(1)}\dots v_{\sigma(r-1)}\ox o(\dots,\hat{w_l},\dots,v_{\sigma(r)},u_1,\dots)\ox w_l v_{\sigma(r+1)}\dots v_{\sigma(t)}\\
-\sum_{\sigma,l\in\unl{s}} (-1)^{s(z-1)+(z-1)+l+s}\xi_{r+1}^{s,t+1,z-1}v_{\sigma(1)}\dots v_{\sigma(r)}\ox o(w_1,\dots,\hat{w_l},\dots,w_z,w_l,u_1,\dots,u_s)\ox v_{\sigma(r+1)}\dots v_{\sigma(t)}\\
\ea
\]
(reduce signs)
\[
\ba{l}
=-\sum_{\sigma,l\in\unl{s}}(-1)^{sz+l}\xi_r^{s-1,t,z}v_{\sigma(1)}\dots v_{\sigma(r-1)}u_l\ox o(w_1,\dots,w_z,v_{\sigma(r)},u_1,\dots,\hat{u_l},\dots,u_s)\ox v_{\sigma(r+1)}\dots v_{\sigma(t)}\\
+\sum_{\sigma,l\in\unl{s}} (-1)^{sz+l}\xi_{r+1}^{s-1,t+1,z}v_{\sigma(1)}\dots v_{\sigma(r)}\ox o(w_1,\dots,w_z,u_l,u_1,\dots,\hat{u_l},\dots,u_s)\ox v_{\sigma(r+1)}\dots v_{\sigma(t)}\\
+\sum_{\sigma,l\in\unl{s}} (-1)^{sz+l}r\xi_{r+1}^{s-1,t+1,z}v_{\sigma(1)}\dots v_{\sigma(r-1)}u_l\ox o(w_1,\dots,w_z,v_{\sigma(r)},u_1,\dots,\hat{u_l},\dots,u_s)\ox v_{\sigma(r+1)}\dots v_{\sigma(t)}\\
+\sum_{\sigma,l\in\unl{s}} (-1)^{sz+l}(t-r+1)\xi_{r}^{s-1,t+1,z}v_{\sigma(1)}\dots v_{\sigma(r-1)}\ox o(w_1,\dots,w_z,v_{\sigma(r)},u_1,\dots,\hat{u_l},\dots,u_s)\ox u_l v_{\sigma(r+1)}\dots v_{\sigma(t)}\\
-\sum_{\sigma,l\in\unl{z}}(-1)^{sz+z+l}\xi^{s,t,z-1}_r v_{\sigma(1)}\dots v_{\sigma(r-1)}\ox o(w_1,\dots,\hat{w_l},\dots,w_z,v_{\sigma(r)},u_1,\dots, u_s)\ox w_lv_{\sigma(r+1)}\dots v_{\sigma(t)}\\
+\sum_{\sigma,l\in\unl{z}} (-1)^{sz+z+l}r\xi_{r+1}^{s,t+1,z-1}v_{\sigma(1)}\dots v_{\sigma(r-1)}w_l\ox o(w_1,\dots,\hat{w_l},\dots ,w_z,v_{\sigma(r)},u_1,\dots,u_s)\ox v_{\sigma(r+1)}\dots v_{\sigma(t)}\\
+\sum_{\sigma,l\in\unl{z}} (-1)^{sz+z+l}(t-r+1)\xi_{r}^{s,t+1,z-1}v_{\sigma(1)}\dots v_{\sigma(r-1)}\ox o(w_1,\dots,\hat{w_l},\dots, w_z,v_{\sigma(r)},u_1,\dots, u_s)\ox w_l v_{\sigma(r+1)}\dots v_{\sigma(t)}\\
+\sum_{\sigma,l\in\unl{z}} (-1)^{sz+z+l}\xi_{r+1}^{s,t+1,z-1}v_{\sigma(1)}\dots v_{\sigma(r)}\ox o(w_1,\dots,\hat{w_l},\dots,w_z,w_l,u_1,\dots,u_s)\ox v_{\sigma(r+1)}\dots v_{\sigma(t)}\\
\ea
\]
(compare terms with $o(w_1,\dots,w_z,u_1,\dots,u_s)$)
\[
\ba{l}
=-\sum_{\sigma,l\in\unl{s}}(-1)^{sz+l}\xi_r^{s-1,t,z}v_{\sigma(1)}\dots v_{\sigma(r-1)}u_l\ox o(w_1,\dots,w_z,v_{\sigma(r)},u_1,\dots,\hat{u_l},\dots,u_s)\ox v_{\sigma(r+1)}\dots v_{\sigma(t)}\\
-\sum_{\sigma} (-1)^{sz}s\xi_{r+1}^{s-1,t+1,z}v_{\sigma(1)}\dots v_{\sigma(r)}\ox o(w_1,\dots,w_z,u_1,\dots,u_s)\ox v_{\sigma(r+1)}\dots v_{\sigma(t)}\\
+\sum_{\sigma} (-1)^{sz}z\xi_{r+1}^{s,t+1,z-1}v_{\sigma(1)}\dots v_{\sigma(r)}\ox o(w_1,\dots,w_z,u_1,\dots,u_s)\ox v_{\sigma(r+1)}\dots v_{\sigma(t)}\\
+\sum_{\sigma,l\in\unl{s}} (-1)^{sz+l}r\xi_{r+1}^{s-1,t+1,z}v_{\sigma(1)}\dots v_{\sigma(r-1)}u_l\ox o(w_1,\dots,w_z,v_{\sigma(r)},u_1,\dots,\hat{u_l},\dots,u_s)\ox v_{\sigma(r+1)}\dots v_{\sigma(t)}\\
+\sum_{\sigma,l\in\unl{s}} (-1)^{sz+l}(t-r+1)\xi_{r}^{s-1,t+1,z}v_{\sigma(1)}\dots v_{\sigma(r-1)}\ox o(w_1,\dots,w_z,v_{\sigma(r)},u_1,\dots,\hat{u_l},\dots,u_s)\ox u_l v_{\sigma(r+1)}\dots v_{\sigma(t)}\\
-\sum_{\sigma,l\in\unl{z}}(-1)^{sz+z+l}\xi^{s,t,z-1}_r v_{\sigma(1)}\dots v_{\sigma(r-1)}\ox o(w_1,\dots,\hat{w_l},\dots,w_z,v_{\sigma(r)},u_1,\dots, u_s)\ox w_lv_{\sigma(r+1)}\dots v_{\sigma(t)}\\
+\sum_{\sigma,l\in\unl{z}} (-1)^{sz+z+l}r\xi_{r+1}^{s,t+1,z-1}v_{\sigma(1)}\dots v_{\sigma(r-1)}w_l\ox o(w_1,\dots,\hat{w_l},\dots ,w_z,v_{\sigma(r)},u_1,\dots,u_s)\ox v_{\sigma(r+1)}\dots v_{\sigma(t)}\\
+\sum_{\sigma,l\in\unl{z}} (-1)^{sz+z+l}(t-r+1)\xi_{r}^{s,t+1,z-1}v_{\sigma(1)}\dots v_{\sigma(r-1)}\ox o(w_1,\dots,\hat{w_l},\dots, w_z,v_{\sigma(r)},u_1,\dots, u_s)\ox w_l v_{\sigma(r+1)}\dots v_{\sigma(t)}.\\
\ea
\]
(combine terms)
\begin{equation}\label{eq:finals2}
\ba{l}
=\sum_{\sigma,l\in\unl{s}}(-1)^{sz+l}(r\xi_{r+1}^{s-1,t+1,z}-\xi_r^{s-1,t,z})v_{\sigma(1)}\dots v_{\sigma(r-1)}u_l\ox o(w_1,\dots,w_z,v_{\sigma(r)},u_1,\dots,\hat{u_l},\dots,u_s)\ox v_{\sigma(r+1)}\dots v_{\sigma(t)}\\
+\sum_{\sigma,l\in\unl{s}} (-1)^{sz+l}(t-r+1)\xi_{r}^{s-1,t+1,z}v_{\sigma(1)}\dots v_{\sigma(r-1)}\ox o(w_1,\dots,w_z,v_{\sigma(r)},u_1,\dots,\hat{u_l},\dots,u_s)\ox u_l v_{\sigma(r+1)}\dots v_{\sigma(t)}\\
+\sum_{\sigma,l\in\unl{z}} (-1)^{sz+z+l}r\xi_{r+1}^{s,t+1,z-1}v_{\sigma(1)}\dots v_{\sigma(r-1)}w_l\ox o(w_1,\dots,\hat{w_l},\dots ,w_z,v_{\sigma(r)},u_1,\dots,u_s)\ox v_{\sigma(r+1)}\dots v_{\sigma(t)}\\
+\sum_{\sigma,l\in\unl{z}}(-1)^{sz+z+l}((t-r+1)\xi_{r}^{s,t+1,z-1}-\xi^{s,t,z-1}_r) v_{\sigma(1)}\dots v_{\sigma(r-1)}\ox o(w_1,\dots,\hat{w_l},\dots,w_z,v_{\sigma(r)},u_1,\dots, u_s)\ox w_lv_{\sigma(r+1)}\dots v_{\sigma(t)}\\
+\sum_{\sigma} (-1)^{sz}(z\xi_{r+1}^{s,t+1,z-1}-s\xi_{r+1}^{s-1,t+1,z})v_{\sigma(1)}\dots v_{\sigma(r)}\ox o(w_1,\dots,w_z,u_1,\dots,u_s)\ox v_{\sigma(r+1)}\dots v_{\sigma(t)}
\ea
\end{equation}
\par

We would like that
\[
(\ref{eq:finals1})+(\ref{eq:finals2})=F_K(1\ox o(u_1,\dots, u_s)\ox v_1\dots v_t\ox o(w_1,\dots, w_z)\ox 1)=0.
\]
To verify this equation it suffices to verify the following:
\begin{enumerate}
\item[lrEQ 1:] $\xi_r^{s,t,z}+r\xi_{r+1}^{s-1,t+1,z}-\xi_r^{s-1,t,z}=0$
\item[lrEQ 2:] $(t-r+1)\xi_r^{s-1,t+1,z}-\xi_r^{s,t,z}=0$
\item[lrEQ 3:] $r\xi_{r-1}^{s,t+1,z-1}-\xi_r^{s,t,z}=0$
\item[lrEQ 4:] $\xi_r^{s,t,z}+(t-r+1)\xi_r^{s,t+1,z-1}-\xi_r^{s,t,z-1}=0$
\item[lrEQ 5:] $\xi_r^{s,t,z}-\xi_{r+1}^{s,t,z}+z\xi_{r+1}^{s,t+1,z-1}-s\xi_{r+1}^{s-1,t+1,z}=0$ for all $0<r<t$.
\item[lrEQ 6:] $z\xi^{s,t+1,z-1}_{t+1}-s\xi^{s-1,t+1,z}_{t+1}+\xi_t^{s,t,z}=0$
\item[lrEQ 7:] $z\xi^{s,t+1,z-1}_1-s\xi^{s-1,t+1,z}_1-\xi_1^{s,t,z}=0$
\end{enumerate}

\begin{proposition}
All of the equations $\mathrm{lrEQ}$ 1--7 are satisfied.
\end{proposition}

\begin{proof}
For lrEQ 1:
\[
r\xi_{r+1}^{s-1,t+1,z}=\frac{(r+z)!(t-r+s-1)!}{(r-1)!(t-r)!(s+t+z)!}
\]
\[
\xi_r^{s-1,t,z}=\frac{(r+z-1)!(t-r+s-1)!}{(r-1)!(t-r)!(s+t+z-1)!}=\frac{(s+t+z)(r+z-1)!(t-r+s-1)!}{(r-1)!(t-r)!(s+t+z)!}.
\]
Whence
\[
\ba{l}
(r-1)!(t-r)!(s+t+z)!(r\xi_{r+1}^{s-1,t+1,z}-\xi_r^{s-1,t,z})\\
=(r+z)!(t-r+s-1)!-(s+t+z)(r+z-1)!(t-r+s-1)!\\
=(r+z-1)!(t-r+s-1)!(r+z-(s+t+z))\\
=(r+z-1)!(t-r+s-1)!(r-s-t)\\
=-(r+z-1)!(t-r+s-1)!(t-r+s)\\
=-(r+z-1)!(t-r+s)!.
\ea
\]
Whence
\[
r\xi_{r+1}^{s-1,t+1,z}-\xi_r^{s-1,t,z}=\frac{-(r+z-1)!(t-r+s)!}{(r-1)!(t-r)!(s+t+z)!}=-\xi_r^{s,t,z}
\]
and
\[
\xi_r^{s,t,z}+r\xi_{r+1}^{s-1,t+1,z}-\xi_r^{s-1,t,z}=0.\ \ \surd
\]
\par

For lrEQ 2:
\[
(t-r+1)\xi^{s-1,t+1,z}_r=\frac{(t-r+1)(r+z-1)!(t-r+s)!}{(r-1)!(t-r+1)!(s+t+z)!}=\frac{(r+z-1)!(t-r+s)!}{(r-1)!(t-r)!(s+t+z)!}=\xi_r^{s,t,z}.
\]
So
\[
(t-r+1)\xi^{s-1,t+1,z}_r-\xi_r^{s,t,z}=0.\ \ \surd
\]
\par

For lrEQ 3:
\[
r\xi^{s,t+1,z-1}_{r+1}=\frac{(r+z-1)!(t-r+s)!}{(r-1)!(t-r)!(s+t+z)!}=\xi_r^{s,t,z}
\ \ \Rightarrow\ \ r\xi^{s,t+1,z-1}_{r+1}-\xi_r^{s,t,z}=0.\ \ \surd
\]
\par

For lrEQ 4:
\[
(t-r+1)\xi_r^{s,t+1,z-1}=\frac{(r+z-2)!(t-r+s+1)!}{(r-1)!(t-r)!(s+t+z)!}
\]
and
\[
\xi_r^{s,t,z-1}=\frac{(r+z-2)!(t-r+s)!}{(r-1)!(t-r)!(s+t+z-1)!}=\frac{(s+t+z)(r+z-2)!(t-r+s)!}{(r-1)!(t-r)!(s+t+z)!}.
\]
Thus
\[
\ba{l}
(r-1)!(t-r)!(s+t+z-1)!\big((t-r+1)\xi_r^{s,t+1,z-1}-\xi_r^{s,t,z-1}\big)\\
=(r+z-2)!(t-r+s+1)!-(s+t+z)(r+z-2)!(t-r+s)!\\
=(r+z-2)!(t-r+s)!(t-r+s+1-(s+t+z))\\
=(r+z-2)!(t-r+s)!(-r+1-z)\\
=-(r+z-2)!(t-r+s)!(r+z-1)\\
=-(r+z-1)!(t-r+s)!
\ea
\]
and
\[
(t-r+1)\xi_r^{s,t+1,z-1}-\xi_r^{s,t,z-1}=-\frac{(r+z-1)!(t-r+s)!}{(r-1)!(t-r)!(s+t+z-1)!}=-\xi_r^{s,t,z}
\]
\[
\Rightarrow \xi_r^{s,t,z}+(t-r+1)\xi_r^{s,t+1,z-1}-\xi_r^{s,t,z-1}=0.\ \ \surd
\]

For lrEQ 5:
\[
z\xi_{r+1}^{s,t+1,z-1}=\frac{z(r+z-1)!(t-r+s)!}{r!(t-r)!(s+t+z)!}
\]
and
\[
s\xi_{r+1}^{s-1,t+1,z}=\frac{s(r+z)!(t-r+s-1)!}{r!(t-r)!(s+t+z)!}
\]
and
\[
\xi_{r+1}^{s,t,z}=\frac{(r+z)!(t-r+s-1)!}{r!(t-r-1)!(s+t+z)!}=\frac{(t-r)(r+z)!(t-r+s-1)!}{r!(t-r)!(s+t+z)!}
\]
Then
\[
\ba{l}
r!(t-r)!(s+t+z)!(-\xi_{r+1}^{s,t,z}+z\xi_{r+1}^{s,t+1,z-1}-s\xi_{r+1}^{s-1,t+1,z})\\
=-(t-r)(r+z)!(t-r+s-1)!+z(r+z-1)!(t-r+s)!-s(r+z)!(t-r+s-1)!\\
=-(r+z-1)!(t-r+s-1)!((t-r)(r+z)-z(t-r+s)+s(r+z))\\
=-(r+z-1)!(t-r+s-1)!((t-r)(r+z)-z(t-r)-zs+s(r+z))\\
=-(r+z-1)!(t-r+s-1)!((t-r)r+sr)\\
=-r(r+z-1)!(t-r+s-1)!(t-r+s)\\
=-r(r+z-1)!(t-r+s)!.
\ea
\]
and
\[
-\xi_{r+1}^{s,t,z}+z\xi_{r+1}^{s,t+1,z-1}-s\xi_{r+1}^{s-1,t+1,z}=-\frac{(r+z-1)!(t-r+s)!}{(r-1)!(t-r)!(s+t+z)!}=-\xi_r^{s,t,z}
\]
\[
\Rightarrow \xi_r^{s,t,z}-\xi_{r+1}^{s,t,z}+z\xi_{r+1}^{s,t+1,z-1}-s\xi_{r+1}^{s-1,t+1,z}=0.\ \ \surd
\]

For lrEQ 6:
\[
z\xi_{t+1}^{s,t+1,z-1}-s\xi_{t+1}^{s-1,t+1,z}=\frac{z(t+z-1)!s!-s(t+z)!(s-1)!}{t!(s+t+z)!}
\]
so that
\[
\ba{l}
t!(s+t+z)!(z\xi_{t+1}^{s,t+1,z-1}-s\xi_{t+1}^{s-1,t+1,z})\\
=z(t+z-1)!s!-s(t+z)!(s-1)!\\
=(t+z-1)!(s-1)!(zs-s(t+z))\\
=-st(t+z-1)!(s-1)!\\
=-t(t+z-1)!s!.
\ea
\]
This implies
\[
z\xi_{t+1}^{s,t+1,z-1}-s\xi_{t+1}^{s-1,t+1,z}=-\frac{t(t+z-1)!s!}{t!(s+t+z)!}=\frac{-(t+z-1)!s!}{(t-1)!(s+t+z)!}=-\xi^{s,t,z}_t
\]
and
\[
\xi^{s,t,z}_t+z\xi_{t+1}^{s,t+1,z-1}-s\xi_{t+1}^{s-1,t+1,z}=0.\ \ \surd
\]

For lrEQ 7:
\[
\ba{l}
t!(s+t+z)!(\xi^{s,t+1,z-1}_1-s\xi_1^{s-1,t+1,z})=z(z-1)!(t+s)!-s(z)!(t+s-1)!\\
=(z-1)!(t+s-1)!(z(t+s)-sz)\\
=(z-1)!(t+s-1)!zt\\
=z!(t+s-1)!t
\ea
\]
and
\[
\xi^{s,t+1,z-1}_1-s\xi_1^{s-1,t+1,z}=\frac{z!(t+s-1)!}{(t-1)!(s+t+z)!}=\xi_1^{s,t,z}
\]
so that
\[
\xi^{s,t+1,z-1}_1-s\xi_1^{s-1,t+1,z}-\xi_1^{s,t,z}=0.\ \ \surd
\]
\end{proof}

We have now checked all possible cases and find that, indeed, $d\phi+\phi d=F_K$,
concluding the proof of Proposition~\ref{prop:phimap},
and therefore also of its restatement Lemma~\ref{lemma:phimap}.


\bibliographystyle{abbrv}

\def\cprime{$'$}

\end{document}